\documentclass[11pt]{article}%
\usepackage{graphics,mathrsfs}
\usepackage{epsfig}
\usepackage{graphicx}
\usepackage{amssymb}
\usepackage{amsmath,bm}
\usepackage{caption}
\usepackage{subcaption}
\usepackage{subfloat}
\usepackage{bbold}
\usepackage{comment}

\newtheorem{theorem}{\bf Theorem}[section]
\newtheorem{proposition}{\bf Proposition}[section]

\newtheorem{remark}{\bf Remark}[section]
\newtheorem{definition}{\bf Definition}[section]

\newtheorem{example}{\bf Example }[section]

\usepackage{color}

\newenvironment{proof}{
\begin{trivlist}
\item[\hspace{\labelsep}{\bf\noindent Proof. }] }{\par\hfill\end{trivlist}
\par}

\date{\empty}

\title{
\huge\bf Fractional generalized cumulative entropy and its dynamic version$^{\small \star}$
\date{\empty}
}

\author{
\large \bf Antonio Di Crescenzo\footnote{
Dipartimento di Matematica, Universit\`a degli Studi di Salerno, Via Giovanni Paolo II, 132, 84084 Fisciano (SA), Italy, 
Email:  adicrescenzo@unisa.it 
}
\qquad
Suchandan Kayal\footnote{
Department of Mathematics, National Institute of Technology Rourkela, 
Rourkela-769008, India, Email: kayals@nitrkl.ac.in,~suchandan.kayal@gmail.com
}
\qquad
Alessandra Meoli\footnote{
Dipartimento di Matematica, Universit\`a degli Studi di Salerno, Via Giovanni Paolo II, 132, 84084 Fisciano (SA), Italy, 
Email:  ameoli@unisa.it 
}
\\
\\
\\
\centerline{$(\star)$ \ paper accepted for publication on} 
\\
\centerline{\bf Communications in Nonlinear Science and Numerical Simulation}
}
\begin{document}
 
\maketitle

\begin{abstract} 
Following the theory of information measures based on the cumulative distribution function, 
we propose the fractional generalized cumulative entropy, and its dynamic version. These entropies 
are particularly suitable to deal with distributions satisfying the proportional reversed hazard model. 
We study the connection with fractional integrals, and some bounds and comparisons based on 
stochastic orderings, that allow to show that the proposed measure is actually a variability measure. 
The investigation also involves various notions of reliability theory, since the considered 
dynamic measure is a suitable extension of the mean inactivity time. 
We also introduce the empirical generalized fractional cumulative entropy as a non-parametric 
estimator of the new measure. It is shown that the empirical measure converges to the proposed 
notion almost surely. 
Then, we address the stability of the empirical measure and provide 
an example of application to real data. Finally,
a central limit theorem is  established under the exponential distribution.

\medskip\noindent
\noindent{\bf Keywords:}  
Cumulative entropy, fractional calculus, stochastic orderings,  estimation.
\end{abstract}

\section{Introduction and background}\label{sec:intro}
Let $X$ be a discrete random variable taking values in $\{x_i; i=1,\ldots,n\}$ and having 
probability mass function $p_{i}=\mathbb P(X=x_{i})$, where $0< p_{i}<1$ for $i=1,\ldots, n.$
The  entropy of $X$ is given by (see Shannon \cite{shannon1948note}) 
\begin{eqnarray}\label{eq1.1}
S(X)=-\sum_{i=1}^{n}p_{i}\ln p_{i},
\end{eqnarray}
where `$\ln$' denotes natural logarithm. It is well known that the entropy (\ref{eq1.1}) 
quantifies the uncertainty contained in the probability distribution associated to $X$, and is particularly important  
in coding theory (see Cover and Thomas \cite{CoverThomas} for specific details). 
In $2009$, Ubriaco \cite{ubriaco2009entropies} extended the notion of entropy to the following version 
based on fractional calculus:  
$$
 S^{\alpha}(X)=\sum_{i=1}^{n}p_{i}(-\ln p_{i})^{\alpha},
 \qquad 0\le \alpha\le1.
$$
Clearly, for $\alpha=1$ the fractional entropy $S^{\alpha}(X)$ reduces to the classical  entropy $(\ref{eq1.1})$. 
The author established that the fractional entropy is stable in the sense of Lesche and thermodynamic 
stability criteria. Moreover, the fractional entropy is nonadditive, positive and concave in nature. 
Machado \cite{machado2014fractional} showed that for the description of a complex system, 
the fractional entropy is quite appealing since it allows high sensitivity to the signal evolution. 
Recently, motivated by the fractional entropy, Machado and Lopes \cite{machado2019fractional} 
proposed the similar measure named fractional Renyi entropy and discussed various properties. 
\par
There have been various developments of uncertainty measures in continuous domain too. 
The continuous analogue of (\ref{eq1.1}) is known as the differential entropy. 
For a nonnegative absolutely continuous random variable $X$ with 
probability density function (PDF) $f(x)$, the differential entropy is
\begin{eqnarray}\label{eq1.2}
H(X)=-\int_{0}^{\infty}f(x)\ln f(x)\,dx.
\end{eqnarray}
However, differently from the classical entropy that takes nonnegative values in the case of discrete
random variables, the differential entropy may assume negative values. For example, for the random variable 
uniformly distributed in the interval $[0,a]$, the differential entropy is equal to $\ln a$, and then is negative 
for $0<a<1$. In order to avoid this fact, various different measures have been proposed in the recent past. 
Indeed, Rao {\em et al.}\  \cite{rao2004cumulative} introduced a measure of uncertainty similar to $H(X)$, 
for which the PDF is replaced by the survival function $\bar{F}(x)=\mathbb{P}(X>x)$ 
in the right-hand-side of (\ref{eq1.2}). This is known as the cumulative residual entropy; for a 
nonnegative random variable it is defined as 
\begin{eqnarray}
 CRE(X)=-\int_{0}^{\infty}\bar{F}(x)\ln\bar{F}(x)\,dx
 \label{eq:CREX}
\end{eqnarray}
and assumes nonnegative values. Along this line Di Crescenzo and Longobardi \cite{di2009cumulative} 
proposed and studied a similar measure, named cumulative entropy. This is defined in terms of the cumulative 
distribution function (CDF) $F(x)=\mathbb{P}(X\leq x)$, i.e.\ (see also Navarro {\em et al.}\ \cite{Navarro2010})
\begin{eqnarray}\label{eq1.3}
CE(X)=-\int_{0}^{l}F(x)\ln F(x)\,dx,
\end{eqnarray}
where $(0,l)$ is the support of $X$.  
The corresponding dynamic measure for the past lifetime is based on the conditional distribution function 
$\mathbb{P}(X\leq x\,|\, X\leq t)=\frac{F(x)}{F(t)}$, $0\leq x\leq t$, and is named cumulative past entropy: 
\begin{eqnarray}\label{eq1.3t}
CE(X;t)=-\int_{0}^{t}\frac{F(x)}{F(t)}\ln\frac{ F(x)}{F(t)}\,dx,
\end{eqnarray}
for $t\in (0,l)$. 
Recently, stimulated by the purpose of constructing a fractional version of $CRE(X)$, in analogy with $S^{\alpha}(X)$  
the following measure has been introduced in Xiong {\em et al.}\ \cite{xiong2019fractional}: 
\begin{eqnarray*} 
 {\cal E}_{\alpha}(X)=\int_{0}^{\infty}\bar{F}(x)[-\ln\bar{F}(x)]^{\alpha}\,dx, \qquad 0\le \alpha\le1.
\end{eqnarray*}
This is called fractional cumulative residual entropy of $X$. 
Among the results on this measure presented in \cite{xiong2019fractional} we mention the asymptotics of its 
empirical version and suitable applications to financial data. 
\par
We remark that a better correspondence with other useful measures can be obtained by including a further term. 
Namely, for any nonnegative random variable $X$ one can consider the {\em fractional generalized cumulative residual entropy}, 
defined as 
\begin{eqnarray}\label{eq1.5}
CRE_{\alpha}(X):=\frac{1}{\Gamma(\alpha+1)}\,{\cal E}_{\alpha}(X)
=\frac{1}{\Gamma(\alpha+1)}\int_{0}^{\infty}\bar{F}(x)[-\ln\bar{F}(x)]^{\alpha}\,dx, 
\qquad   \alpha\ge 0.
\end{eqnarray}
If $\alpha$ is a positive integer, say $\alpha=n\in \mathbb N$, then $CRE_{n}(X)$ identifies with the generalized cumulative residual entropy, 
that has been introduced by Psarrakos and Navarro \cite{PsarrakosNavarro2013}. We recall that, if $\alpha=n\in\mathbb N$, then 
$CRE_{n}(X)$ is a dispersion measure
that is strictly related to the (upper) record values of a sequence of independent and identically distributed random
variables. It is also related to the relevation transform and to the interepoch intervals of a non-homogeneous Poisson process 
(see, for instance, Toomaj and Di Crescenzo \cite{ToDiCr2020} and references therein for some recent results on this measure). 
\par
Along the lines of the above mentioned researches, in this paper we propose a fractional version of the generalized cumulative 
entropy. The new measure is defined similarly as in (\ref{eq1.5}), by replacing the survival function with the CDF of $X$. 
Fractional versions of various information measures have been proposed in the recent years. Indeed, more advanced 
mathematical tools are suitable to handle complex systems and anomalous dynamics. Various characteristics of fractional 
calculus allow the related measures to better capture long-range phenomena and nonlocal dependence in certain random 
systems. For instance, we recall the recent papers by Zhang and Shang \cite{zhang2019uncertainty} and 
Wang and Shang \cite{wang2020complexity}, finalized to study new fractional modifications on the discrete version 
of the cumulative residual entropy (\ref{eq:CREX}), which are useful to analyze time series and have been applied in the 
context of data from the stock market. Further applications of multiscale fractional measures to the analysis of time series 
has been successfully exploited in Dong and Zhang \cite{Dong2020}. Different information measures based on fractional 
calculus have been investigated in Yu {\em et al.}\ \cite{Yu2012}, where the fractional entropy and other related notions 
have been obtained by replacing the Riemann integral with the Riemann-Liouville integral operator, leading to new tools 
of interest in image analysis.  
\par
It is worth mentioning that the fractional measure proposed in this paper is particularly suitable to be adopted in the 
context of the proportional reversed hazard model, as already seen for various information measures derived from the 
cumulative entropy (\ref{eq1.3}). Moreover, it exhibits a nice connection with the Riemann-Liouville fractional integral with 
respect to another function.
\par
The rest of the paper is organized as follows. In Section $2$, we discuss some properties of the fractional generalized 
cumulative entropy and provide some examples from typical distributions of interest. 
In particular, we 
show that the fractional generalized cumulative entropy is actually a variability measure. 
Moreover,  
we analyze the proposed measure under the proportional reversed hazard model. 
We also point out the above mentioned connection with fractional integrals of the Riemann-Liouville type. 
Section $3$ is devoted to various bounds for the proposed measures. Some comparisons are 
also studied by means of suitable stochastic orderings. In Section $4$, we propose the dynamic version of the considered 
measure. We provide some examples satisfying the proportional reversed hazard model and arising from the 
analysis of first-hitting time distributions in customary stochastic processes. In Section $5$, we propose a 
non-parametric estimator of the new measure. We discuss its statistical characteristics, with special care on the 
asymptotic properties. Such properties allow the empirical measure to be successfully adopted to describe 
the information content in experimental data, such as for  time-series and in signal analysis. 
Accordingly,  the section  
includes 
an example of application to a real dataset. 
Finally, some final remarks complete the paper in Section 6. 
\par
Throughout the paper,  $\mathbb N$  denotes the set of positive integers, and $\mathbb N_0=\mathbb N\cup\{0\}$. 
Moreover, aiming to provide suitable comparisons we shall deal with the stochastic orders recalled hereafter.  
Let $X$ and $Y$ be random variables with CDFs $F$ and $G$, respectively. Then, $X$ is smaller than $Y$
\par\noindent
-- \ in the {\em usual stochastic order}, denoted by $X\le_{st}Y$, 
if $F(x)\ge G(x)$ for all $x \in \mathbb R$; 
\par\noindent
-- \ in the {\em dispersive order}, denoted by $X\le_{d}Y$, 
if $F^{-1}(v)-F^{-1}(u)\le G^{-1}(v)-G^{-1}(u)$ for all $0<u\le v<1$, 
where $F^{-1}$ and $G^{-1}$ denote the right-continuous inverses of $F$ and $G$, respectively;
\par\noindent
-- \  in the {\em hazard rate order}, denoted by $X\leq_{hr} Y$, 
if $\bar{G}(x)/\bar{F}(x)$ is nondecreasing with respect to $x$, where $\bar{F}=1-F$ and $\bar{G}=1-G$ 
are respectively the survival functions of $X$ and $Y$. 
\par
We  refer the reader to the book of Shaked and Shanthikumar \cite{Shaked} for their main properties.
\section{Fractional generalized cumulative entropy}
Let $X$ be a nonnegative random variable with support $(0,l)$ and CDF $F$. 
Then, in analogy with the measure given in  (\ref{eq1.5}), the 
{\em fractional generalized cumulative entropy} of $X$ is defined by
\begin{eqnarray}\label{eq2.1}
 CE_{\alpha}(X)
 =\frac{1}{\Gamma(\alpha+1)}\int_{0}^{l}F(x)[-\ln F(x)]^{\alpha}\,dx,
 \qquad   \alpha>0,
\end{eqnarray}
provided that the integral in the right-hand-side is finite. 
Clearly, for $\alpha=1$, the measure $CE_{\alpha}(X)$ reduces to the cumulative
entropy given in (\ref{eq1.3}). From (\ref{eq2.1}) it is not hard to see that 
$$
 \lim_{\alpha\to 0^+} CE_{\alpha}(X)
 =\left\{
 \begin{array}{ll}
 l-\mathbb E(X), & 0<l<+\infty\\
 + \infty, & l=+\infty.
 \end{array}
 \right.
$$ 
The fractional generalized cumulative entropy is nonnegative and nonadditive. If $0<\alpha<1$, 
then it is concave with respect to the distribution function. 
From (\ref{eq2.1}) it is not hard to see that $CE_{\alpha}(X)\geq 0$; moreover one has  $CE_{\alpha}(X)= 0$ if and only if $X$ 
is degenerate. 
We recall that if $X$ is absolutely continuous, then the function in the square brackets in the right-hand side of Eq.\ (\ref{eq2.1}) 
corresponds to the cumulative reversed hazard rate function of $X$, i.e.
$$
 T(x)=-\ln F(x)=\int_x^l \tau(x)\,dx, \qquad 0<x<l,
$$  
where $\tau(x)=\frac{f(x)}{F(x)}$ is the reversed hazard rate of $X$ and $f(x)$ is the PDF of $X$. 
\par
The fractional generalized cumulative entropy is provided in Table \ref{table:examples} 
for some distributions, where $\Gamma$ denotes the complete gamma function, and where 
\begin{equation}
 E_{\alpha}(\beta)=\int_1^{\infty} e^{- \beta t } t^{-\alpha} \,dt
 \label{eq:expintf} 
\end{equation}
is the exponential integral function. For these cases, 
$CE_{\alpha}(X)$ is decreasing in $ \alpha$ and tends to 0 as $ \alpha\to \infty$. 
\par
When $\alpha$ is a positive integer, say $n\in \mathbb N$, then $CE_{n}(X)$ identifies with  the generalized cumulative 
entropy, defined and studied by Kayal \cite{kayal2016generalized}. In this case, $CE_{n}(X)$ is strictly related 
to the lower records of a sequence of i.i.d.\ random variables, and to the recursive reversed relevation transform 
(see also Di Crescenzo and Toomaj \cite{di2017further}). 
\begin{table}[t]
\begin{center}
\begin{tabular}{llll}
\hline
{} & $\!\!\!\!\!\!\!\!\!\!\!\!$distribution  &  $F(x)$     &  $CE_{\alpha}(X),  \quad   \alpha>0$      \\
\hline\\[-3mm]
(i) & \hbox{uniform} 
& $\displaystyle\frac{x}{l}$, \quad $0\leq x\leq l$  
& 
$\displaystyle\frac{l}{2^{\alpha+1}}$ 
\\[3mm]
\hline\\[-3mm]
(ii) & \hbox{power} 
& $\left(\displaystyle\frac{x}{l}\right)^b$, \quad $0\leq x\leq l$, \quad $b>0$  
& 
$ \displaystyle\frac{l\,b^{\alpha}}{(b+1)^{\alpha+1}}$ 
\\[3mm]
\hline\\[-3mm]
(iii) & \hbox{Fr\'echet} 
& $e^{- b \, x^{-\eta}}$, \quad $x>0$, \quad $b,\eta>0\quad$  
& $\displaystyle\frac{b^{1/\eta}}{\eta\, \Gamma(\alpha+1)}\,\Gamma\Big(\alpha-\displaystyle\frac{1}{\eta}\Big), 
\quad 0<\frac{1}{\eta}<\alpha $
\\[3mm]
\hline\\[-3mm]
(iv) & ${\hbox{bounded} \atop  \hbox{Fr\'echet}}$
& $\exp\left\{ b  (1- \frac{l}{x} )\right\}$,   \quad $0<x\leq l$, \quad $b>0\;$  
& $\displaystyle\frac{l \,b}{ \alpha} \left( e^{b}  (\alpha+b) E_{\alpha}(b)-1\right)$
\\[3mm]
\hline
\end{tabular}
\end{center}
\caption{
Fractional generalized cumulative entropy for some distributions.
}
\label{table:examples}
\end{table}
\begin{remark}\label{rem:symmCE}
It is not hard to see that if $X$ has bounded support $(0, l)$ and possesses a symmetric distribution, such that 
$F(x)=\bar F(l-x)$ for all $0\leq x\leq l$, then from Eqs.\ (\ref{eq1.5}) and (\ref{eq2.1}) one has 
$$
 CE_{\alpha}(X)=CRE_{\alpha}(X) \qquad \hbox{for all $\alpha>0$.}
$$ 
For instance, this is true for the uniform distribution (cf.\ Case (i) of Table \ref{table:examples} and Example 1 of \cite{xiong2019fractional}).  
\end{remark}
\par
We point out that, even though the cases of interest usually deal with absolutely continuous random variables, 
the fractional generalized cumulative entropy can also refer to discrete random variables. 
For instance, if $U$ is uniformly distributed on  $\{1,2,\ldots,n\}$ then from (\ref{eq2.1}) 
we have 
\begin{equation}
 CE_{\alpha}(U)
 =\frac{1}{\Gamma(\alpha+1)} \sum_{k=1}^{n-1}\left( \frac{k}{n}\right)  \left[-\ln\left( \frac{k}{n}\right) \right]^{\alpha},
 \qquad   \alpha>0.
 \label{eq:CEU}
\end{equation}
\par
Various information measures already known are expressed   
as the expectation of a given function of the random variable of interest. 
For instance, we recall  that for the cumulative residual entropy (\ref{eq:CREX}) 
one has (cf.\ Theorem 2.1 of Asadi and Zohrevand \cite{Asadi2007})
$$
 CRE(X)=\mathbb E[{\rm mrl}(X)],
$$
where, for all $t\geq 0$ such that $\bar F (t)>0$, 
\begin{equation}
 {\rm mrl}(t):=\mathbb E[X-t | X>t]
 =\frac{1}{\bar F(t)}\int_t^{\infty} \bar F (x)\, dx
 \label{eq:mrlt}
\end{equation}
is the mean residual life of a nonnegative lifetime $X$. 
Similarly, for the cumulative entropy (\ref{eq1.3}) we have  (see Theorem 3.1 of \cite{di2009cumulative}) 
$$
 CE(X)= \mathbb E[\tilde \mu (X)],
$$
where, for all $t\geq 0$ such that $F (t)>0$, 
\begin{equation}
 \tilde \mu(t) :=\mathbb E[t-X | X\leq t]
 =\frac{1}{F(t)} \int_{0}^{t}F(x)\,dx
 \label{eq:tildemu}
\end{equation}
is the mean inactivity time of $X$, and deserves interest in reliability theory. Hereafter we state 
a similar result for the fractional generalized cumulative entropy, by
expressing it as the expectation of a decreasing function of $X$, for $\alpha$ fixed. 
\begin{proposition}\label{prop:CEXxi}
Let $X$ be a nonnegative random variable with support $(0, l)$ and CDF $F$. If 
\begin{equation}
 \xi_\alpha(x):=\frac{1}{\Gamma(\alpha+1)} \int_{x}^{l}[-\ln F(t)]^{\alpha}\,dt<\infty, \qquad \alpha>0,
 \label{eq:defxial}
\end{equation}
then 
\begin{equation}
 CE_{\alpha}(X)=\mathbb E[\xi_\alpha(X)],
 \qquad \alpha>0.
 \label{eq:CEXxial}
\end{equation}
\end{proposition} 
\begin{proof}
From (\ref{eq2.1}), by Fubini's  theorem and (\ref{eq:defxial}) we obtain
\begin{eqnarray*}
 CE_{\alpha}(X) \!\!\!\! 
 &=&\!\!\!\!  \frac{1}{\Gamma(\alpha+1)}\int_{0}^{l}[-\ln F(t)]^{\alpha}\,dt \int_0^t dF(x) 
 \\
 &=&\!\!\!\!  \frac{1}{\Gamma(\alpha+1)}\int_{0}^{l} \left( \int_x^l  [-\ln F(t)]^{\alpha}\,dt \right)dF(x)
 =  \int_{0}^{l} \xi_\alpha(x) \, dF(x),
\end{eqnarray*}
so that the relation (\ref{eq:CEXxial}) holds. 
\end{proof}
\par
We remark that for $\alpha=1$,  the result given in Proposition \ref{prop:CEXxi} corresponds to 
Proposition 3.1 of \cite{di2009cumulative}. 
\par
We recall that the measure defined in (\ref{eq1.5}) is shift-independent. That is, for the affine transformation 
$Y=cX+b$, $c>0$, $b\ge0$, one has 
$$
 CRE_{\alpha}(Y)= CRE_{\alpha}(cX+b)=c\;CRE_{\alpha}(X)\qquad \hbox{for all $\alpha>0$.} 
$$
Hereafter we show that 
the same property holds for the fractional generalized cumulative entropy (\ref{eq2.1}). The proof is a straightforward 
consequence of the relation $F_{Y}(x)=F_{X}(\frac{x-b}{c})$, $x\in\mathbb{R}$, and thus is omitted. 
\begin{proposition}\label{prop2.1}
Let $Y=cX+b$, where $c>0$ and $b\ge 0$. Then, 
$$
 CE_{\alpha}(Y)=CE_{\alpha}(cX+b)=c \;CE_{\alpha}(X)
 \qquad \hbox{for all $ \alpha>0$.}
$$
\end{proposition}
\begin{remark}
It is worth mentioning that the fractional generalized cumulative entropy is actually a {\em variability 
measure} (following Bickel and Lehmann \cite{BickelLehmann1979}), thanks to previously 
given results. Indeed, under suitable assumptions the following properties hold: 
\\
(P1) \ $CE_{\alpha}(X+b)=CE_{\alpha}(X)$ for all constants $b$,
\\
(P2) \ $CE_{\alpha}(cX)=c\,CE_{\alpha}(X)$ for all  $c>0$,
\\
(P3) \ $CE_{\alpha}(a)= 0$ for any degenerate random variable at $a$,
\\
(P4) \ $CE_{\alpha}(X)\geq  0$ for all $X$,
\\
(P5) \ $X\leq_d Y$ implies $CE_{\alpha}(X)\leq CE_{\alpha}(Y)$. 
\\
Property (P5) is proved in Theorem \ref{th2.1} below. 
\end{remark}
\par
We point out that the results mentioned above, in particular that the fractional generalized cumulative entropy 
is a variability measure, can be seen to hold even if $X$ has more general support, say $(a,b)$. 
For instance, it is not hard to see that if Eq.\ (\ref{eq2.1}) is replaced by  
$CE_{\alpha}^{(a,b)}(X)=\frac{1}{\Gamma(\alpha+1)}\int_{a}^{b}F(x)[-\ln F(x)]^{\alpha}\,dx$, $\alpha>0$, 
with $-\infty\leq a<b\leq \infty$, then we come to a suitable extension of the considered measure. 
We leave the details to the reader, being straightforward. 
\par
In analogy with the normalized cumulative entropy proposed in \cite{di2009cumulative}, it is possible to define 
a normalized version of $CE_{\alpha}(X)$. Let us assume that the cumulative entropy (\ref{eq1.3}) 
is finite and non-zero. Then, the {\em normalized fractional generalized cumulative entropy} of a nonnegative 
random variable $X$ with support $(0,l)$ is defined as 
\begin{eqnarray}\label{eq2.6}
NCE_{\alpha}(X)=\frac{CE_{\alpha}(X)}{(CE(X))^{\alpha}}
=\frac{1}{\Gamma(\alpha+1)}
\frac{\int_{0}^{l}F(x)[-\ln F(x)]^{\alpha}\,dx}
{\left(\int_{0}^{l}F(x)[-\ln F(x)]\,dx\right)^{\alpha}},\qquad  \alpha>0.
\end{eqnarray}
Clearly, one has
$$
 \lim_{\alpha \to 0^+} NCE_{\alpha}(X)= \int_{0}^{l}F(x) \,dx,
 \qquad 
 \lim_{\alpha \to 1} NCE_{\alpha}(X)=1.
$$
For instance, Table \ref{table:examplesNCE} provides the normalized fractional generalized cumulative entropy    
for some distributions. The corresponding plots are given in Figure \ref{fig:Figure1}. 
\begin{table}[t]
\begin{center}
\begin{tabular}{lll}
\hline
{} &  distribution    &     $NCE_{\alpha}(X),  \quad   \alpha>0$      \\
\hline\\[-3mm]
(i) & \hbox{uniform}
& 
$ \Big(\displaystyle \frac{l}{2}\Big)^{1-\alpha}$ 
\\[3mm]
\hline\\[-3mm]
(ii) &  \hbox{power}
& 
$ \Big(\displaystyle \frac{l}{b+1}\Big)^{1-\alpha}$ 
\\[3mm]
\hline\\[-3mm]
(iii) &  \hbox{Fr\'echet}
& $\displaystyle\frac{  b^{(1-\alpha)/\eta} \,\Gamma\left(\alpha- \frac{1}{\eta}\right)}
{  \eta^{1-\alpha}\, \Gamma(\alpha + 1)\left(\Gamma\left(1- \frac{1}{\eta}\right)\right)^{\alpha}},
\quad 0<\displaystyle\frac{1}{\eta}<\alpha\leq 1$
\\[6mm]
\hline\\[-3mm]
(iv) & ${\hbox{bounded} \atop  \hbox{Fr\'echet}}$
&   $\displaystyle \frac{ (l \,b)^{1-\alpha} }{ \alpha}\, 
\frac{ e^{b}  (\alpha+b) E_{\alpha}(b)-1}{\left( e^{b}  (1+b) E_{1}(b)-1\right)^{\alpha}}$ 
\\[3mm]
\hline
\end{tabular}
\end{center}
\caption{
Normalized fractional generalized  cumulative entropy for the same distributions of Table \ref{table:examples}. 
}
\label{table:examplesNCE}
\end{table}
%
%
\begin{figure}[t]
\centering 
 \includegraphics[scale=0.45]{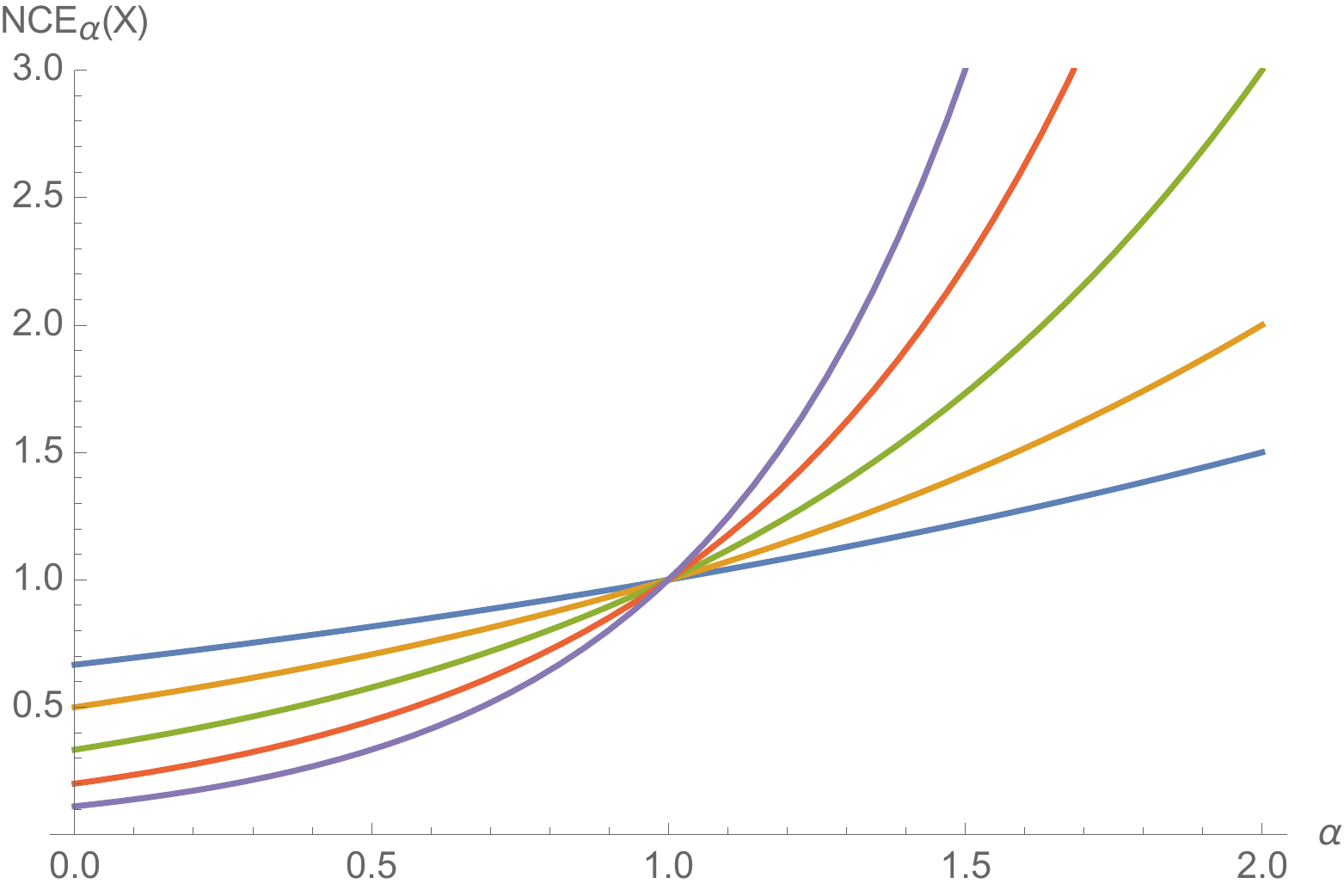}
 \;
 \includegraphics[scale=0.45]{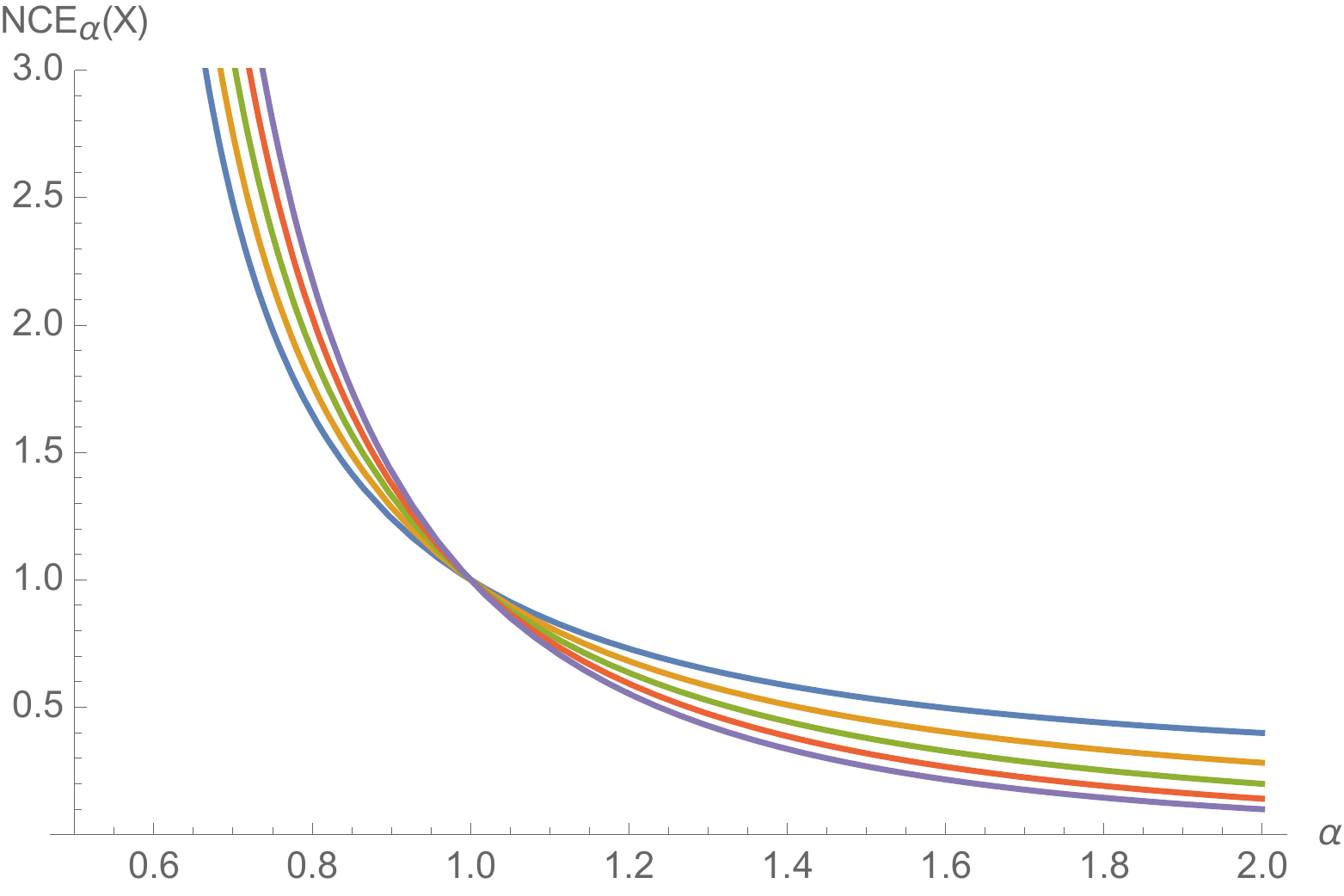}
\caption{
The normalized fractional cumulative entropy for the distributions given in case (ii) of Table \ref{table:examplesNCE}, 
with $l=1$, on the left, and in case (iii) of Table \ref{table:examplesNCE}, with $\eta=1$, on the right, as a function of $\alpha$, 
for $b=0.5$, $1$, $2$, $4$, $8$, from top to bottom near the origin in the case on the left, and viceversa on the right. 
}
\label{fig:Figure1}
\end{figure}
\par
With reference to the relation (\ref{eq2.1}) we also point out that, if $X$ is absolutely continuous with support $(0,l)$, 
CDF $F$ and  PDF $f$, then the fractional generalized cumulative entropy can be expressed as 
\begin{equation}
CE_{\alpha}(X)
=\frac{1}{\Gamma(\alpha+1)}\int_{0}^{1} \frac{u \,(-\ln u)^{\alpha}}{f(F^{-1}(u))}\,du,
\qquad   \alpha>0.
\label{eq:Faltern}
\end{equation}
\begin{example}
Let $X$ have half-logistic distribution distribution on $(0,\infty)$, with PDF and CDF given respectively by 
$$
 f(x)= \frac{4 e^{-2x}}{(1+e^{-2x})^2}, 
 \qquad 
 F(x)= \frac{1-e^{-2x}}{1+e^{-2x}}, 
 \qquad x\geq 0. 
$$
Hence, noting that $f(F^{-1}(u))=1-u^2$, making use of  (\ref{eq:Faltern}) and of the integral representations 
of the Riemann zeta function  (cf.\ Section 23.2 of Abramowitz and Stegun \cite{Abram1992})  
$$
 \zeta (s)=\sum_{k=1}^{\infty} k^{-s}, \qquad s>1,
$$
in this case one has
$$
 CE_{\alpha}(X)
 = \frac{\zeta (\alpha+1)}{2^{\alpha+1}},
 \qquad   \alpha>0.
$$
\end{example}
%
\subsection{Proportional reversed hazard model}
The proportional reversed hazard model is expressed by a nonnegative absolutely continuous random variable 
$ X_{\theta} $ whose CDF is a power of a baseline  function $ F\left(x\right) $, which in turn is the CDF of  
a nonnegative absolutely continuous random variable $X$, i.e.\ (see, for instance, Di Crescenzo \cite{DiCrescenzo2000}, 
Gupta and Gupta \cite{Gupta2007}, and Li {\em et al.}\ \cite{Li2020}) 
\begin{equation}\label{PRHR}
F_{X_{\theta}}(x)=\left[F (x)\right]^{\theta},\qquad x\in\mathbb{R}, \quad \theta>0.
\end{equation}
This model is often encountered in first-hitting-time problems of Markov processes and in the analysis 
of the reliability of parallel systems.  The PDF and the reversed hazard rate function of $ X_{\theta} $ are given respectively by 
\begin{equation}\label{densityPRHM}
f_{X_{\theta}} (x)=\theta\left[F (x)\right]^{\theta-1}f (x )
\end{equation}
and
\begin{equation}\label{ReversedHazardRate}
\tau_{X_\theta} (x)=\frac{f_{X_{\theta}}(x)}{F_{X_{\theta}}(x)}=\theta\,\frac{f(x)}{F(x)},
\end{equation}
where $ f(x) $ is the PDF of the baseline distribution. 
We now evaluate the fractional generalized cumulative entropy of $ X_{\theta} $.
\begin{proposition}
Let $X$ have support $(0,l)$, with $0<l<+\infty$,  or with $l=+\infty$ such that $\lim_{x\to+ \infty}x\,[-\ln F(x)]^{\alpha}=0$. Then, 
under the proportional reversed hazard model (\ref{PRHR}), the fractional generalized cumulative entropy 
of $ X_{\theta} $, $ \theta>0 $, can be expressed as
\begin{equation}\label{FGCETheta}
 CE_{\alpha}\left(X_{\theta}\right)
 =\mathcal{E}_{\theta}(\alpha)-\mathcal{E}_{\theta}(\alpha+1),\qquad  \alpha>0,
\end{equation}
where
\begin{equation}\label{ETheta}
	\mathcal{E}_{\theta}\left(\alpha\right)
	:=\frac{1}{\Gamma\left(\alpha\right)}\,\mathbb{E}\left[X_{\theta}\left[-\ln F_{X_{\theta}}\left(X_{\theta}\right)\right]^{\alpha-1}\right], 
	\qquad  \alpha>0.
\end{equation}
\end{proposition} 
\begin{proof}
Recalling (\ref{eq2.1}) and (\ref{PRHR}), we have
\begin{align*}
CE_{\alpha}\left(X_{\theta}\right)
=\frac{1}{\Gamma\left(\alpha+1\right)}\int_{0}^{l}\left[F\left(x\right)\right]^{\theta}\left[-\ln \left[F\left(x\right)\right]^{\theta}\right]^{\alpha}\,{d}x.
\end{align*}
Integrating by parts, we get
\begin{align}\label{ByParts}
CE_{\alpha}\left(X_{\theta}\right)&=\frac{1}{\Gamma\left(\alpha+1\right)}\left\{-\theta\int_{0}^{l}x\left[F\left(x\right)\right]^{\theta-1}f\left(x\right)\left[-\ln \left[F\left(x\right)\right]^{\theta}\right]^{\alpha}\,{d}x\right.
\nonumber\\
&\quad{}\left.+\alpha\,\theta\int_{0}^{l}x\left[F\left(x\right)\right]^{\theta-1}f\left(x\right)\left[-\ln \left[F\left(x\right)\right]^{\theta}\right]^{\alpha-1}\,{d}x\right\}
\\
&=\frac{1}{\Gamma\left(\alpha+1\right)}\left\{-\int_{0}^{l}xf_{X_{\theta}}\left(x\right)\left[-\ln F_{X_{\theta}}\left(x\right)\right]^{\alpha}\,{d}x\right.\nonumber\\
&\quad{}\left.+\alpha\int_{0}^{l}xf_{X_{\theta}}\left(x\right)\left[-\ln F_{X_{\theta}}\left(x\right)\right]^{\alpha-1}\,{d}x\right\},\nonumber
\end{align}
where the last equality is due to (\ref{densityPRHM}). The final result then follows from (\ref{ETheta}).
\end{proof}
\par
For instance, it is not hard to see that if $F(x)=e^{-x^{-1}}$, $x>0$, then (\ref{ETheta})   yields 
$\mathcal{E}_{\theta}\left(\alpha\right)=\frac{\theta}{\alpha -1}$, for $\alpha>1$, and thus 
(\ref{FGCETheta}) gives $CE_{\alpha}\left(X_{\theta}\right)=\frac{\theta}{\alpha(\alpha -1)}$, $\alpha>1$, this being in 
agreement with case (iii) of Table \ref{table:examples}. 
\par
We observe that $ CE_{\alpha}\left(X_{\theta}\right) $ can be alternatively rewritten as the sum of weighted fractional generalized cumulative entropies. In fact, from (\ref{PRHR}) and (\ref{ReversedHazardRate}), Eq.\ (\ref{ByParts}) becomes
\begin{align*}
CE_{\alpha}\left(X_{\theta}\right)&=\frac{1}{\Gamma\left(\alpha+1\right)}\left\{-\int_{0}^{l}xF_{X_{\theta}}\left(x\right)\tau_{X_\theta}\left(x\right)\left[-\ln F_{X_{\theta}}\left(x\right)\right]^{\alpha}\,{d}x\right.\\
&\quad{}\left.+\alpha \int_{0}^{l}xF_{X_{\theta}}\left(x\right)\tau_{X_\theta}\left(x\right)\left[-\ln F_{X_{\theta}}\left(x\right)\right]^{\alpha-1}\,{d}x\right\}.
\end{align*}
Moreover, the fractional generalized cumulative entropy of $ X_{\theta} $ satisfies a recurrence relation. 
Indeed, from (\ref{FGCETheta}) we have 
\begin{align}\label{RecurrenceRelation}
CE_{\alpha+1}\left(X_{\theta}\right)&=\mathcal{E}_{\theta}\left(\alpha+1\right)-\mathcal{E}_{\theta}\left(\alpha+2\right)
\nonumber\\
&=\mathcal{E}_{\theta}\left(\alpha\right)-\mathcal{E}_{\theta}\left(\alpha+2\right)-CE_{\alpha}\left(X_{\theta}\right).
\end{align}
More generally, in the next proposition we express the fractional generalized cumulative entropy of $ X_{\theta} $ of order $ \alpha+n $ 
in terms of the same measure of order $ \alpha $.
\begin{proposition}
Under the proportional reversed hazard model (\ref{PRHR}), for $ \alpha>0 $ and $ \theta>0 $, and for any integer $ n\geq 2 $,
\begin{equation}\label{RecurrenceN}
CE_{\alpha+n}\left(X_{\theta}\right)
=
\left(-1\right)^{n} CE_{\alpha}\left(X_{\theta}\right) +
\mathcal{E}_{\theta}\left(\alpha+n\right)-\mathcal{E}_{\theta}\left(\alpha+n+1\right)
+ \left(-1\right)^{n-1}
\left[\mathcal{E}_{\theta}\left(\alpha\right)-\mathcal{E}_{\theta}\left(\alpha+1\right)\right].
\end{equation}
%
\end{proposition}
\begin{proof}
	The proof is by induction on $ n $. Let us consider the case $ n=2 $. Applying the recurrence relation (\ref{RecurrenceRelation}) twice yields
	\begin{align*}
	CE_{\alpha+2}\left(X_{\theta}\right)&=\mathcal{E}_{\theta}\left(\alpha+1\right)-CE_{\alpha+1}\left(X_{\theta}\right)-\mathcal{E}_{\theta}\left(\alpha+3\right)\\
	&=CE_{\alpha}\left(X_{\theta}\right)+\mathcal{E}_{\theta}\left(\alpha+2\right)-\mathcal{E}_{\theta}\left(\alpha+3\right)
	-\mathcal{E}_{\theta}\left(\alpha\right)+\mathcal{E}_{\theta}\left(\alpha+1\right),
	\end{align*}
	which coincides with (\ref{RecurrenceN}). Now let us assume that Eq.\ (\ref{RecurrenceN}) holds for some $ n>2 $. Then, due to (\ref{RecurrenceRelation}),
	\begin{align*}
	CE_{\alpha+n+1}\left(X_{\theta}\right)&
	=\mathcal{E}_{\theta}\left(\alpha+n\right)-CE_{\alpha+n}\left(X_{\theta}\right)-\mathcal{E}_{\theta}\left(\alpha+n+2\right)
	\\
	&=\mathcal{E}_{\theta}\left(\alpha+n\right)
	-\left(-1\right)^{n-1}\left[\mathcal{E}_{\theta}\left(\alpha\right)-\mathcal{E}_{\theta}\left(\alpha+1\right)\right]
	+\left(-1\right)^{n-1}CE_{\alpha}\left(X_{\theta}\right)
	\\
	&\quad{}-\mathcal{E}_{\theta}\left(\alpha+n\right)+\mathcal{E}_{\theta}\left(\alpha+n+1\right)-\mathcal{E}_{\theta}\left(\alpha+n+2\right)
	\\
	&=\left(-1\right)^{n+1}CE_{\alpha}\left(X_{\theta}\right)+\mathcal{E}_{\theta}\left(\alpha+n+1\right)-\mathcal{E}_{\theta}\left(\alpha+n+2\right)
	+\left(-1\right)^{n}\left[\mathcal{E}_{\theta}\left(\alpha\right)-\mathcal{E}_{\theta}\left(\alpha+1\right)\right].
	\end{align*}
The validity of Eq.\ (\ref{RecurrenceN}) for $ n $ implies its validity
for $ n + 1 $. Therefore, it is true for all $ n\geq 2 $.
\end{proof}
%
\subsection{Connection with fractional integrals}
Let us now pinpoint the connection between the generalized measures  defined above and some notions of 
fractional calculus.   
\par
The growing interest on the theory and applications of Fractional Calculus has led several authors to introduce 
new notions of fractional integrals. In this area, for instance we refer the reader to the book by 
Samko {\em et al.}\ \cite{Samko_etal1993}. 
We recall that for any sufficiently well-behaved function $\phi$ locally integrable in the interval $I=(a,b]$, 
the (Riemann-Liouville) left- and right-sided fractional integrals of order $\alpha$ of $\phi$, 
for $a<x<b$ and  $\alpha>0$, are defined respectively as 
$$
 I_{a+}^{\alpha} \phi(x) =\frac{1}{\Gamma(\alpha)} \int_a^x \frac{\phi(y)}{(x-y)^{1-\alpha}}\,{ d}y, 
 \qquad
 I_{b-}^{\alpha} \phi(x) =\frac{1}{\Gamma(\alpha)} \int_x^b \frac{\phi(y)}{(y-x)^{1-\alpha}}\,{ d}y.
$$ 
These notions have been extended to the case of  integral with respect to another function. 
Indeed, if $g$ is a 
strictly 
increasing monotone function on $I$, having a continuous derivative 
$g'$ on $(a,b)$, then the left- and right-sided fractional integrals of order $\alpha$ of $\phi$ with respect to $g$, 
for $a<x<b$ and  $\alpha>0$, are given respectively by (see Section 18.2 of  Samko {\em et al.}\ \cite{Samko_etal1993}, 
or Section 2.5 of Kilbas \cite{Kilbas_etal2006})  
\begin{equation}
 I_{a+;g}^{\alpha} \phi(x) =\frac{1}{\Gamma(\alpha)} \int_a^x \frac{g'(y) \phi(y)}{[g(x)-g(y)]^{1-\alpha}}\,{ d}y, 
 \qquad
 I_{b-;g}^{\alpha} \phi(x) =\frac{1}{\Gamma(\alpha)} \int_x^b \frac{g'(y) 
 \phi(y)}{[ g(y)-g(x)]^{1-\alpha}}\,{ d}y.
 \label{eq:Ointegr}
\end{equation}
It is worth mentioning that both the fractional generalized cumulative entropy  and the 
fractional generalized cumulative residual entropy  can be 
expressed in terms of the integrals given in  (\ref{eq:Ointegr}). 
Indeed, from (\ref{eq2.1}) it  is not hard to see that 
$$
 CE_{\alpha}(X)=   \lim_{a\to 0} \lim_{x\to l} I_{a+;g}^{\alpha+1} \phi(x), 
 \qquad   \alpha>0,
$$
for
$$
 g(x)= \ln   F(x), \qquad 
 \phi(x)=[F(x)]^{2}[f(x)]^{-1},
$$
provided that the PDF $f$ is positive, continuous and integrable in $(0, l)$. Similarly, 
under the same assumptions for $f$ on $(0,\infty)$, from (\ref{eq1.5}) one has 
$$
 CRE_{\alpha}(X)=   \lim_{x\to 0} \lim_{b\to \infty} I_{b-;g}^{\alpha+1} \phi(x), 
 \qquad \alpha>0,
$$
for
$$
 g(x)=- \ln \bar F(x), \qquad 
 \phi(x)=[\bar F(x)]^{2}[f(x)]^{-1}.
$$
These remarks justify the fractional nature of the measures introduced so far. 

\section{Bounds and stochastic ordering}
The aim of this section is two-fold:   obtaining some bounds of the proposed fractional
measure, and providing results based on stochastic comparisons.  
\par
The cumulative entropy of the sum of two nonnegative independent random variables is larger 
than the maximum of their individual cumulative entropies (cf.\ \cite{di2009cumulative}). 
Below, we show that a similar inequality holds for the fractional generalized cumulative entropy. 
The proof follows from Theorem $2$ of \cite{rao2004cumulative}, therefore it is omitted.
\begin{proposition}\label{prop3.1}
For any pair of nonnegative absolutely continuous independent random
variables $X$ and $Y$, we have
$$ 
 CE_{\alpha}(X+Y)\ge \max\{CE_{\alpha}(X),CE_{\alpha}(Y)\}.
$$
\end{proposition}
\par
In the following proposition, we obtain a  bound of the fractional entropy (\ref{eq2.1}) 
in terms of the cumulative entropy (\ref{eq1.3}).  
\begin{proposition}\label{prop3.2}
If $X$ is a nonnegative  random variable with support $(0,l)$,  $0<l<\infty$,  and with finite cumulative entropy, then 
$$
 CE_{\alpha}(X)
 \left\{
 \begin{array}{ll}
 \leq \,\displaystyle \frac{l^{1-\alpha}}{\Gamma(\alpha+1)} \,(CE(X))^{\alpha}, \ & \hbox{if  \ }0<\alpha\leq 1\\[4mm]
  \geq \,\displaystyle \frac{l^{1-\alpha}}{\Gamma(\alpha+1)} \,(CE(X))^{\alpha}, \ & \hbox{if  \ } \alpha\geq 1.
  \end{array}
  \right.
$$
\end{proposition}
\begin{proof}
For $0<\alpha\le1$, the CDF of $X$ satisfies $F(x)\le(F(x))^{\alpha}$ for all $x\in (0,l)$, so that Eq.\ (\ref{eq2.1}) yields 
$$
 CE_{\alpha}(X)
\leq \frac{1}{\Gamma(\alpha+1)}\int_{0}^{l}[-F(x)\ln F(x)]^{\alpha}\,dx
=\int_{0}^{l} \varphi_{\alpha}(\eta(x))\,dx,
$$
where 
$$
 \eta(x)=-F(x)\ln F(x)\geq 0,\qquad \varphi_{\alpha}(t)=\frac{t^{\alpha}}{\Gamma(\alpha+1)},
$$
with $\varphi_{\alpha}(t)$ concave in $t\geq 0$, for $0<\alpha\le1$. Hence, from the integral Jensen inequality we obtain 
$$
 CE_{\alpha}(X)
\leq  l  \, \varphi_{\alpha}\left(\frac{1}{l} \int_{0}^{l} \eta(x) \, dx\right)
=\frac{l^{1-\alpha}}{\Gamma(\alpha+1)}\left(-\int_0^l F(x)\ln F(x)\, dx\right)^{\alpha},
$$
this giving the proof due to (\ref{eq1.3}).  The case $\alpha\ge 1$ can be treated similarly. 
\end{proof}
\par
Clearly, under the assumptions of Proposition \ref{prop3.2} one has the following relation for the 
normalized fractional generalized cumulative entropy defined in (\ref{eq2.6}): 
$$
 NCE_{\alpha}(X)
 \left\{
 \begin{array}{ll}
 \leq \,\displaystyle \frac{l^{1-\alpha}}{\Gamma(\alpha+1)}, \ & \hbox{if  \ }0<\alpha\leq 1\\[4mm]
  \geq \,\displaystyle \frac{l^{1-\alpha}}{\Gamma(\alpha+1)}, \ & \hbox{if  \ } \alpha\geq 1.
  \end{array}
  \right.
$$
\par
\par
The next proposition provides some bounds for the fractional generalized cumulative entropy of bounded distributions.
\begin{proposition}\label{prop3.3}
For any random variable $X$ with support $(0,l)$ and with finite $CE_{\alpha}(X)$, for $\alpha>0$, we have
\begin{itemize}
\item[(a)] $CE_{\alpha}(X)\ge D_{\alpha}\,e^{H(X)},$ with $D_{\alpha}=\exp\{\int_{0}^{1}\ln (x(-\ln x)^{\alpha})\,dx\}$ and $H(X)$  given by (\ref{eq1.2});
\item[(b)] $CE_{\alpha}(X)\ge \displaystyle\frac{1}{\Gamma(\alpha+1)} \int_{0}^{l}F(x)[1-F(x)]^{\alpha}\,dx$;
\item[(c)] $CE_{\alpha}(X)\le \displaystyle\frac{l}{\Gamma(\alpha+1)} \left( \frac{\alpha}{e} \right)^{\alpha}$, for $0<\alpha\leq 1$. 
\end{itemize}
\end{proposition}
\begin{proof}
The first inequality can be reached applying the log-sum inequality (see, for instance, \cite{CoverThomas}).
Recalling (\ref{eq2.1}), the second inequality can be obtained by using $\ln u\le u-1$ for $0<u\le1$. 
Similarly, we note that $u(- \log u)^{\alpha}$ is nonnegative and concave in $u\in (0,1)$ for all $\alpha\in (0,1]$, so that 
$u(- \log u)^{\alpha}\leq (- \log \theta)^{\alpha-1}[ \alpha \theta -u (\alpha +\log \theta)]$ 
for all $u\in (0,1)$, with $\alpha\in (0,1]$ and  $\theta\in (0,1]$. Hence, from (\ref{eq2.1}) we have 
$$
 CE_{\alpha}(X)\leq  \frac{1}{\Gamma(\alpha+1)} (- \log \theta)^{\alpha-1}
 \left\{ \alpha \theta l -[l- \mathbb{E}(X)] (\alpha +\log \theta)\right\}.
$$
By taking $\theta=e^{-\alpha}$ we finally obtain the third inequality. 
\end{proof}
\par
For $\alpha=1$, the results (a) and (b) of Proposition \ref{prop3.3} become the relations given in 
Propositions 4.2 and 4.3 of \cite{di2009cumulative},  respectively. Clearly, by resorting to Fubini's theorem 
the inequality given in (b) can be expressed as 
$$
 CE_{\alpha}(X)\ge \mathbb E[\psi_{\alpha}(X)], \qquad \hbox{where \quad 
 $\psi_{\alpha}(x)=\frac{1}{\Gamma(\alpha+1)} \int_{x}^{l} [1-F(y)]^{\alpha}\,dy, \quad \alpha>0$.}
$$
%
%
\subsection{Some comparisons}
Let us now present some ordering properties of the fractional generalized cumulative
entropy. We refer to 
the stochastic orders recalled in Section \ref{sec:intro}. 
\par
In the following example we show that in general the usual stochastic ordering does not imply the ordering 
of fractional generalized cumulative entropies.
\begin{example}\label{ex3.1*}
Consider two random variables having power distribution, with CDF $F(x)=(x/l)^{b}$ and $G(x)=(x/l)^{d}$, 
where $0\le x\le l$ and $b,d>0$. Further, for $b\le d$ we have $X\le_{st}Y.$ 
Moreover, recalling (ii) of Table \ref{table:examples}, 
\begin{eqnarray*}
CE_{\alpha}(X)=\frac{l \,b^{\alpha}}{(b+1)^{\alpha+1}}, \qquad 
CE_{\alpha}(Y)=\frac{l \,d^{\alpha}}{(d+1)^{\alpha+1}}, \qquad \alpha >0.
\end{eqnarray*}
However, for some values of the parameters and some choices of $\alpha$ the condition 
$CE_{\alpha}(X)\le CE_{\alpha}(Y)$ is not fulfilled (see Figure \ref{fig:Pdfderived}).
%
\begin{figure}[t]

\center
\includegraphics[height=1.8in]{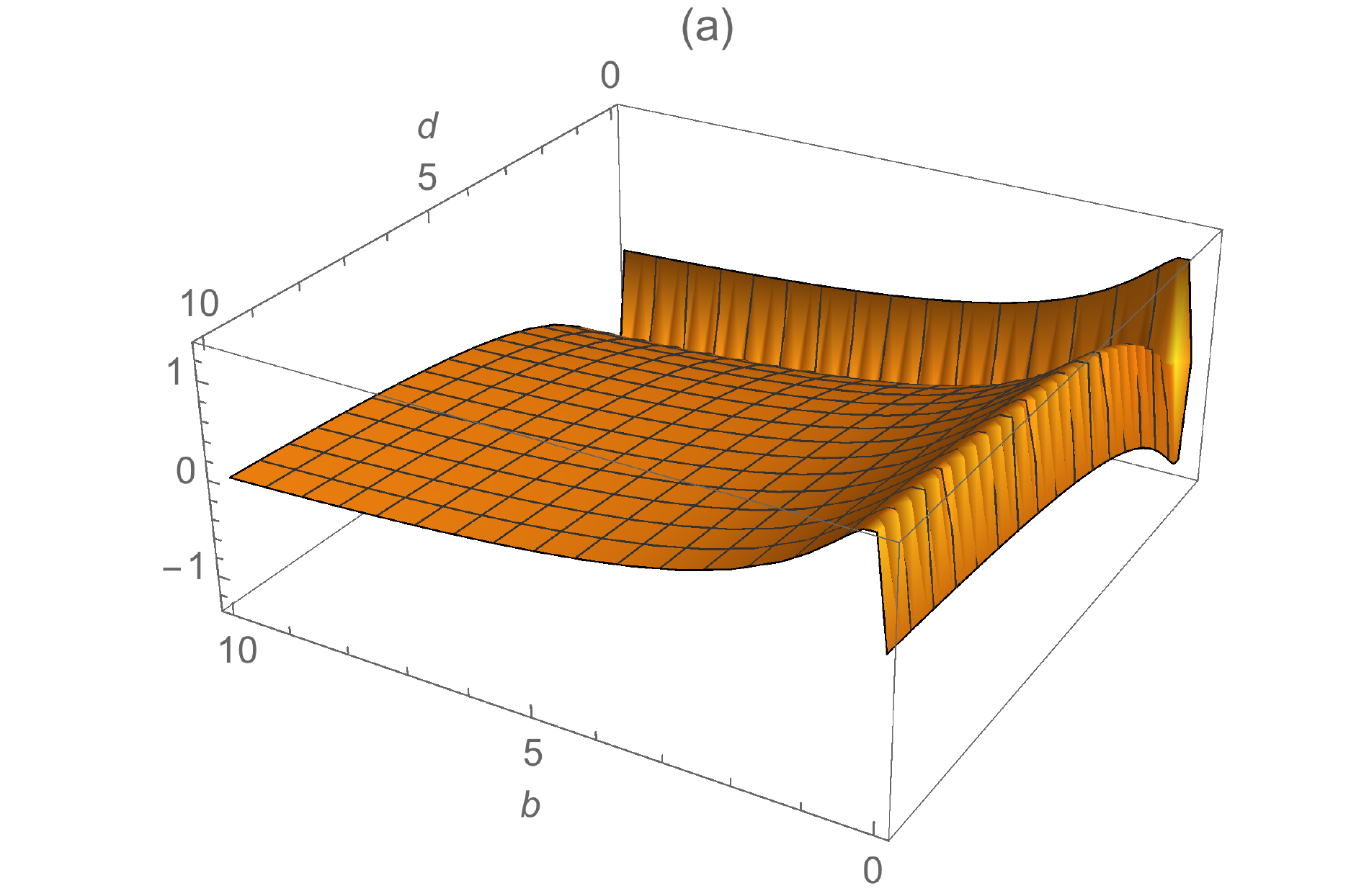}
\;\;
\includegraphics[height=1.8in]{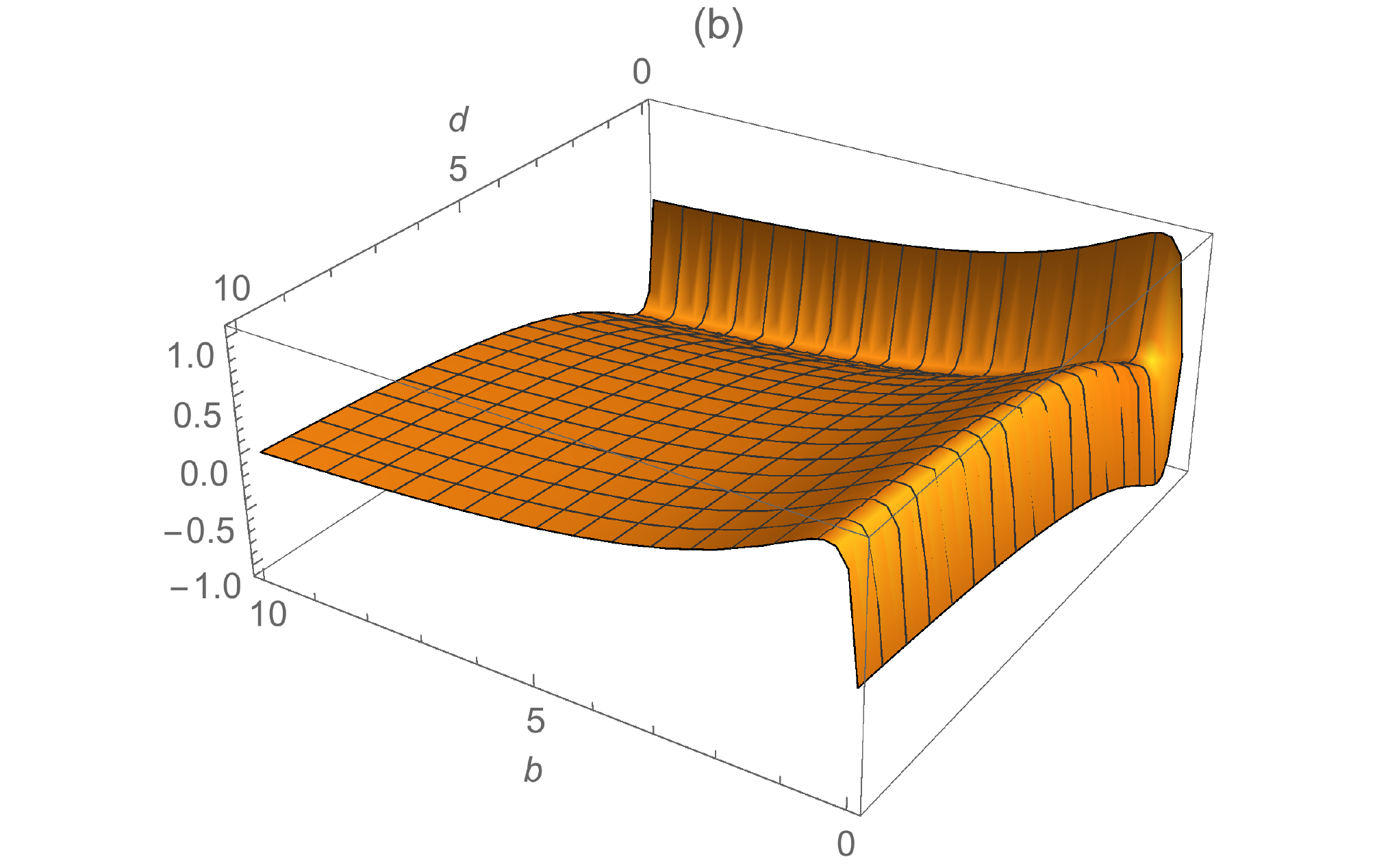}
\\[2mm]
\includegraphics[height=1.8in]{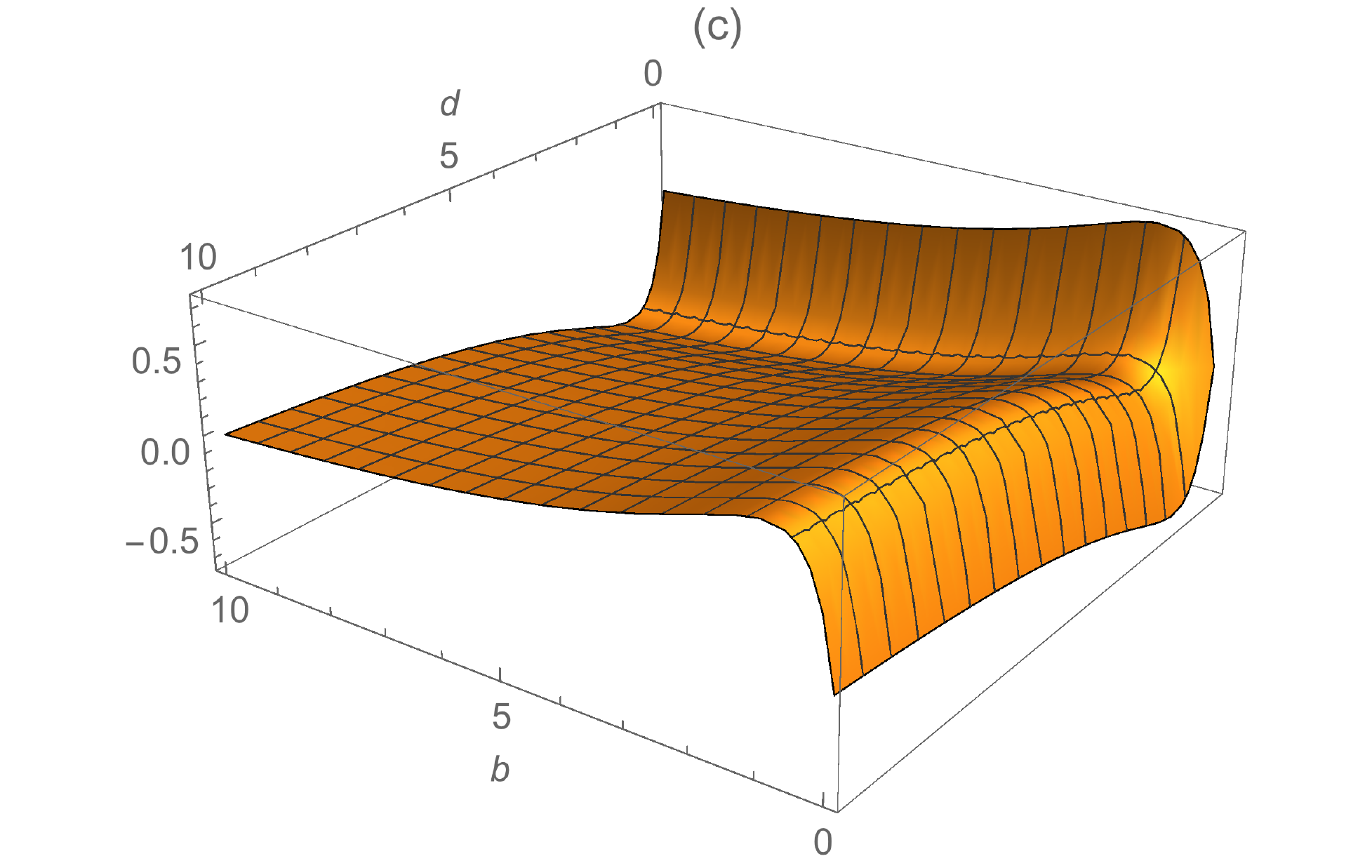}
\;\;
\includegraphics[height=1.55in]{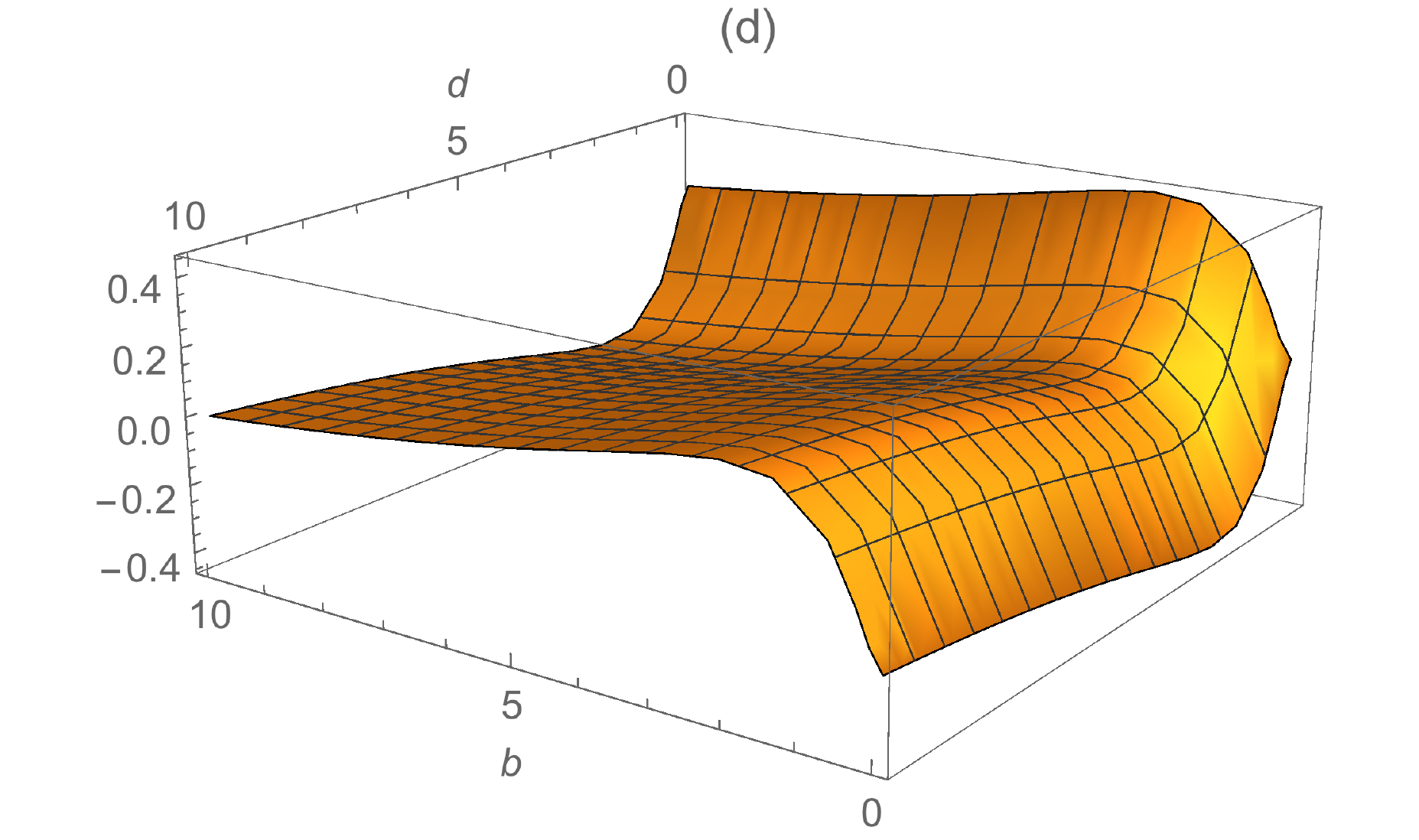}
\caption{
Graphs of $CE_{\alpha}(X)-CE_{\alpha}(Y)$ as considered in Example \ref{ex3.1*} for $b,d\in(0,10)$, $l=3$,  and 
(a) $\alpha=0.25$, (b) $\alpha=0.5$, (c) $\alpha=1$, (d) $\alpha=2$. 
}
\label{fig:Pdfderived}
\end{figure}
\end{example}
\par
Now, we obtain some stochastic ordering properties of the considered measure. 
We show that more dispersed random
variables produce larger fractional generalized cumulative entropies. 
\begin{theorem}\label{th2.1}
Let $X$ and $Y$ be  nonnegative absolutely continuous random variables with 
PDF's $f$ and $g$, and CDF's $F$ and $G$, respectively. Then, $X\le_{d}Y$ implies that 
$CE_{\alpha}(X)\le CE_{\alpha}(Y)$, for all $\alpha>0$.
\end{theorem}
\begin{proof}
From the representation given in (\ref{eq:Faltern}), for $\alpha>0$ one has 
\begin{eqnarray*}
CE_{\alpha}(X)-CE_{\alpha}(Y)
= \frac{1}{\Gamma(\alpha+1)} \int_{0}^{1} u\,(-\ln u)^{\alpha}
\left[\frac{1}{f(F^{-1}(u))}-\frac{1}{g(G^{-1}(u))}\right]du.
\end{eqnarray*}
The thesis then immediately follows recalling that $X\le_{d}Y$ if and only if 
$f(F^{-1}(u)) \geq g(G^{-1}(u))$ for all $u\in (0,1)$ (see Section 3.B of \cite{Shaked}). 
\end{proof}
\par
The following  comparison result involves the hazard rate order. Moreover, we recall that 
$X$ is said to be decreasing failure rate (DFR) if  $\bar{F}$ is logconvex.
\begin{theorem}
Let the random variables $X$ and $Y$ satisfy the same assumptions of Theorem \ref{th2.1}. 
Further, assume that $X\le_{hr}Y$ and let $X$ or $Y$ be DFR. 
Then, we have $CE_{\alpha}(X)\le CE_{\alpha}(Y)$ for all $\alpha>0$. 
\end{theorem}
\begin{proof}
The proof follows from Theorem 2.1(b) of Bagai and Kochar
\cite{bagai1986tail} and the result given in Theorem \ref{th2.1}. 
\end{proof}
\par
In various applied contexts it is appropriate to compare random measures arising from possibly ordered systems. 
Let us then face the following problem: to express the fractional generalized cumulative entropy of $X$ in terms of 
suitable quantities depending on $X$ and $Y$, where the latter random variables are ordered in the 
usual stochastic order sense.  
\begin{proposition}
Let $X$ and $Y$ be  nonnegative random variables  with finite but unequal means, 
with CDF's $F$ and $G$ respectively, and such that 
$X\le_{st}Y$  or  $Y\le_{st} X$. If condition (\ref{eq:defxial}) holds and if $\mathbb E[\xi_{\alpha}(Y)]<\infty$, then for $\alpha>0$
\begin{eqnarray}
 CE_{\alpha}(X)=\mathbb E[\xi_{\alpha}(Y)]+\mathbb E[\xi'_{\alpha}(Z)]\,[\mathbb E(X)- \mathbb E(Y)], 
\end{eqnarray}
where $Z$ is an absolutely continuous nonnegative random variable with PDF
$$
 f_{Z}(x)=\frac{ G(x)-F(x)}{\mathbb E(X)- \mathbb E(Y)},
 \qquad x>0.
$$
\end{proposition}
\begin{proof}
Recalling Eq.\  (\ref{eq:CEXxial}), i.e.\   $CE_{\alpha}(X)=\mathbb E[\xi_\alpha(X)]$, the proof follows from the 
probabilistic analogue of the mean value theorem given in Theorem $4.1$ of Di Crescenzo \cite{di1999probabilistic}.
\end{proof}
\par 
We remark that, due to (\ref{eq:defxial}), one has 
$\xi'_{\alpha}(x)= - \frac{1}{\Gamma(\alpha+1)} [-\ln F(x)]^{\alpha}  \leq 0$ for all $x> 0$, and $\alpha>0$. 

\section{Dynamic version of fractional generalized cumulative entropy}\label{sect:dynam}
In this section we develop a dynamic version of the fractional generalized cumulative entropy. 
To this aim we take as reference a notion from reliability theory. Suppose that a system, started at time 0, 
is seen to be failed at a pre-specified inspection time, say $t\in (0,l)$. In this case the uncertainty relies on the past, 
in the sense that the unknown system failure instant has occurred in $(0,t)$, previous than the inspection time.  
Let $X$ be the random variable that denotes the failure instant. We can consider the fractional generalized 
cumulative entropy of the past time 
$$
 X_{(t)}:=[X\,|\,X\leq t], \qquad t\in (0,l).
$$ 
Various measures have been proposed in the literature for $X_{(t)}$, such as the cumulative past entropy given in 
(\ref{eq1.3t}). Indeed, one has $CE(X;t)=CE(X_{(t)})$, for $t\in (0,l)$. 
Here we can define the {\em dynamic fractional generalized cumulative entropy}, for $\alpha>0$, as 
\begin{eqnarray}\label{eq2.1*}
 CE_{\alpha}(X;t):=CE_{\alpha}(X_{(t)})
 = \frac{1}{\Gamma(\alpha+1)} \int_{0}^{t}\frac{F(x)}{F(t)}
 \left[-\ln\left(\frac{F(x)}{F(t)}\right)\right]^{\alpha} dx, 
 \qquad  t\in (0,l).
\end{eqnarray}
Clearly, if $\alpha\to 1$ then $CE_{\alpha}(X;t)$ tends to cumulative past entropy $CE(X;t)$ given in (\ref{eq1.3t}). 
Moreover, from (\ref{eq2.1*}) we have that $CE_{\alpha}(X;t)$ reduces to the fractional generalized  cumulative entropy  
(\ref{eq2.1}) when $t\to l$. 
\begin{example}\label{ex:CEXt}
(i) Let $X$ have power distribution in the interval $(0,l)$ with parameter $b$, 
as in the case (ii) of Table \ref{table:examples}. 
Then, the dynamic fractional generalized cumulative entropy is given by 
$$ 
 CE_{\alpha}(X;t)=\frac{t\,b^{\alpha} }{ (b+1)^{\alpha+1}},  \qquad t\in (0,l).
$$ 
(ii) Let $X$ have Fr\'echet distribution with parameters $b$ and 1, i.e.\ $F(x)=e^{- b \, x^{-1}}$, $x>0$, with $b>0$. Then, 
from (\ref{eq2.1*}) we have 
\begin{equation}
 CE_{\alpha}(X;t)=\frac{b}{\alpha}\left( \Big( \alpha+\frac{b}{t}\Big) E_{\alpha}\Big(\frac{b}{t} \Big)e^{ b t}-1\right), 
 \qquad t>0,
 \label{eq:CEalphaXt}
\end{equation}
where $E_{\alpha}$ is defined in (\ref{eq:expintf}). In this case, 
some plots of $CE_{\alpha}(X;t)$ are given in Figure \ref{fig:Figuretre}.
%
\begin{figure}[t]
\centering 
 \includegraphics[scale=0.45]{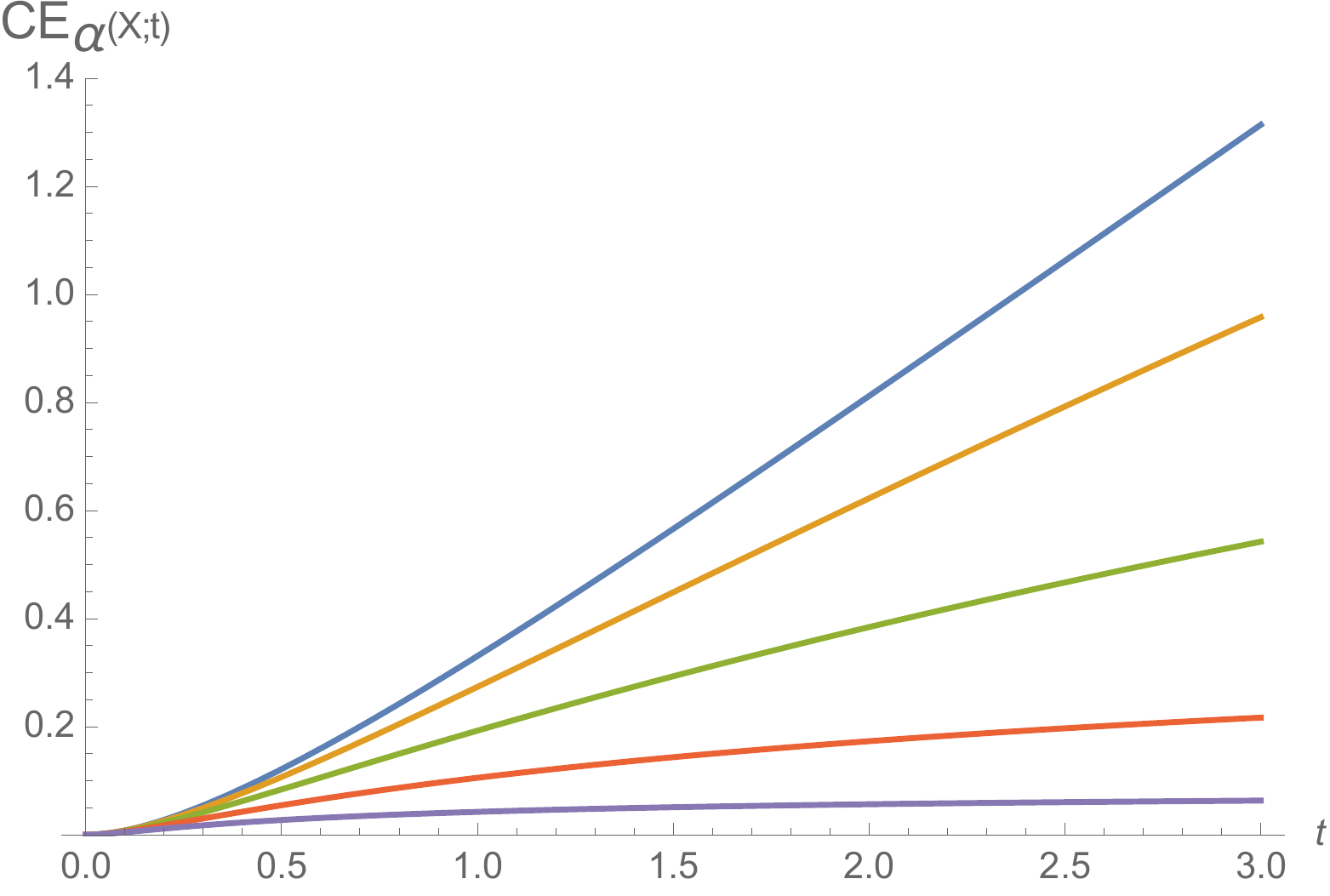}
 \;
 \includegraphics[scale=0.45]{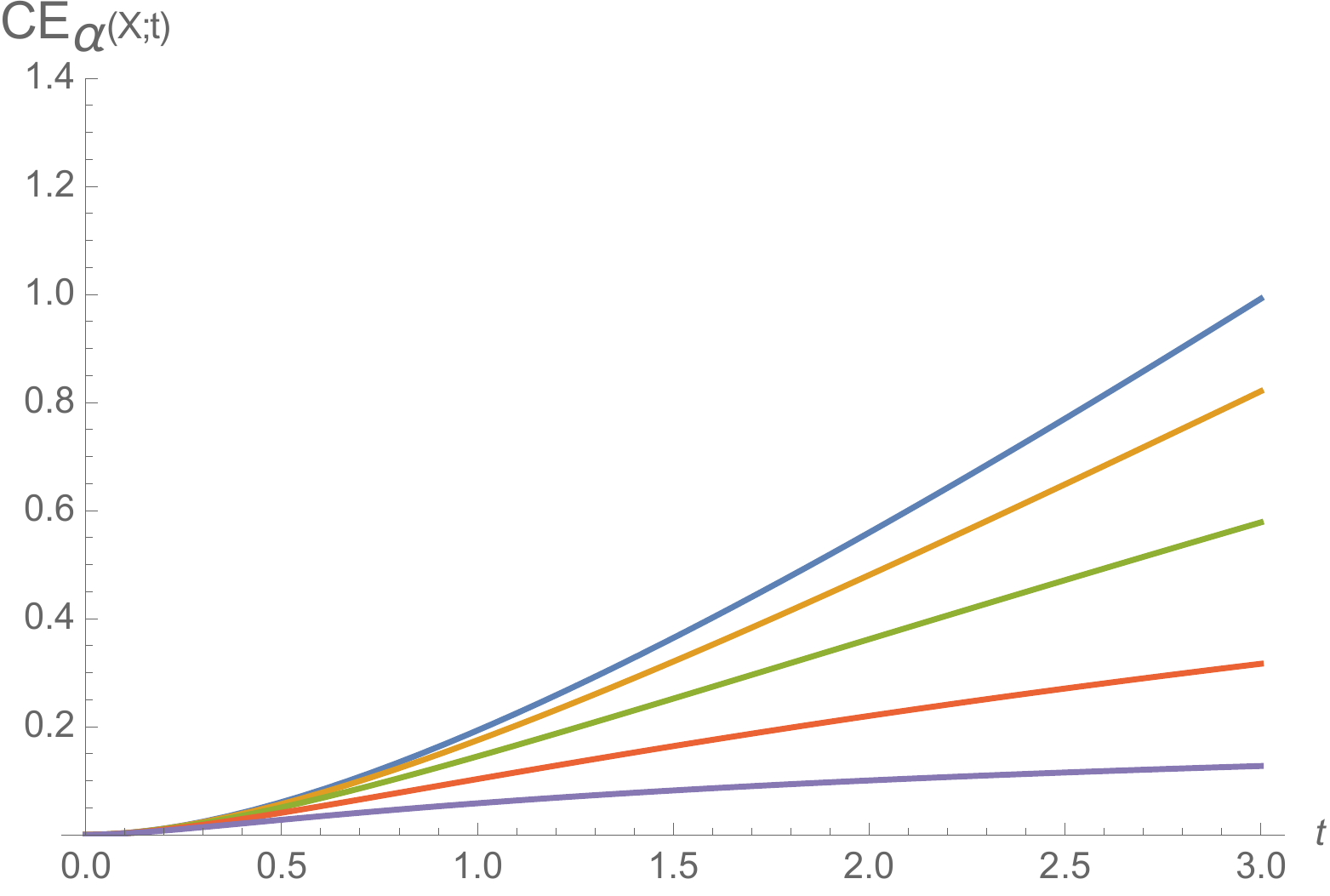}
\caption{
The dynamic fractional generalized cumulative entropy (\ref{eq:CEalphaXt}) of the Fr\'echet distribution with parameters $b$ and 1, 
for $\alpha=0.25$, $0.5$, $1$, $2$, $4$ (from top to bottom), with $b=1$ (left) and $b=3$ (right). 
}
\label{fig:Figuretre}
\end{figure}
\end{example}
\par
Similarly to Proposition \ref{prop2.1}, we get the following result concerning the effect of an affine transformation.
\begin{proposition}\label{prop2.3}
Let $Y=cX+b$, where $c>0$ and $b\ge0.$ Then, $CE_{\alpha}(cX+b;t)=c\, CE_{\alpha}\left(X;\frac{t-b}{c}\right)$ 
for all $\alpha >0$ and all $t\in(b,b+c\,l)$.
\end{proposition}
\begin{remark}
Under the same conditions specified in Remark \ref{rem:symmCE}, from Eq.\ (\ref{eq2.1*}) 
one can verify that for a symmetric distribution the following relation holds: 
$$
 CE_{\alpha}(X;t)=CRE_{\alpha}(X;l-t)
  \qquad \hbox{for all $\alpha>0$ \ and \ all $t\in (0,l)$,}
$$
where, in analogy with (\ref{eq1.5}), 
\begin{equation}
 CRE_{\alpha}\left(X;t\right):= \frac{1}{\Gamma(\alpha+1)}
 \int_{t}^{l}\frac{\bar{F}(x)}{\bar{F}(t)} \left[- \ln\left(\frac{\bar{F}(x)}{\bar{F}(t)}\right)\right]^{\alpha} dx, 
 \qquad t\in (0,l)
 \label{eq:CREalphaXt}
\end{equation}
is the dynamic fractional generalized cumulative residual entropy of $X$. 
\end{remark}
\par
It is worth mentioning that the dynamic measures $CE_{\alpha}(X;t)$ and $CRE_{\alpha}\left(X;t\right)$ not only 
provide respectively a generalization of the cumulative past entropy and of the cumulative residual entropy, attained in the limit  
$\alpha \to 1$. They also constitute a further extension of well-known measures of interest in reliability theory. 
Indeed, from Eqs.\ (\ref{eq2.1*}) and (\ref{eq:CREalphaXt}) we have respectively 
$$
  \lim_{\alpha\to 0^+} CE_{\alpha}(X;t)=\tilde \mu(t), 
  \qquad t\in (0,l),
$$
and 
$$
   \lim_{\alpha\to 0^+} CRE_{\alpha}(X;t)= {\rm mrl}(t), 
  \qquad t\in (0,l),
$$
where $\tilde \mu (t)$ is the  mean inactivity time (\ref{eq:tildemu}), and 
where ${\rm mrl}(t)$ is the mean residual life of $X$, defined in (\ref{eq:mrlt}).  
\par
Similarly to Proposition \ref{prop3.2}, we obtain the following bounds for the dynamic measure defined in (\ref{eq2.1*}), 
for $t\in (0,l)$:
$$
CE_{\alpha}\left(X; t\right)   
\left\{
 \begin{array}{ll}
 \leq \, \displaystyle \frac{t^{1-\alpha}}{\Gamma(\alpha+1)} \,(CE(X;t))^{\alpha}, \ & \hbox{if  \ }0<\alpha\leq 1\\[4mm]
  \geq \, \displaystyle \frac{t^{1-\alpha}}{\Gamma(\alpha+1)} \,(CE(X;t))^{\alpha}, \ & \hbox{if  \ } \alpha\geq 1,
  \end{array}
  \right.
$$
with $CE(X;t)$ given in (\ref{eq1.3t}). Moreover, following the same arguments of the proof of Proposition \ref{prop3.3} 
we obtain the following bounds for the dynamic fractional generalized cumulative entropy. The proof is omitted being similar.
\begin{proposition}
For any random variable X with support $(0, l)$ and with finite $CE_{\alpha}(X;t)$, for $\alpha > 0$ and $t\in (0,l)$ we have 
the following bounds:
\begin{itemize}
\item[(a)] $CE_{\alpha}(X;t)\ge D_{\alpha} \,e^{H(X;t)},$ where $D_{\alpha}=\exp\{\int_{0}^{1}\ln (x(-\ln x)^{\alpha}) \,dx\}$ 
and $H(X;t)$  is the cumulative past entropy (\ref{eq1.3t});
\item[(b)] $CE_{\alpha}(X;t)\ge \displaystyle \frac{1}{\Gamma(\alpha+1)} \int_{0}^{t}\frac{F(x)}{F(t)}\left(1-\frac{F(x)}{F(t)}\right)^{\alpha}dx$.
\item[(c)] $CE_{\alpha}(X;t)\le \displaystyle\frac{t}{\Gamma(\alpha+1)} \left( \frac{\alpha}{e} \right)^{\alpha}$, for $0<\alpha\leq 1$. 
\end{itemize}
\end{proposition}
\par
Next, we introduce the class of distributions based on the monotonicity property of the dynamic fractional generalized cumulative entropy. 
It was proved in \cite{di2009cumulative} that the class of distributions having decreasing dynamic cumulative entropy is empty. 
A similar property holds for $CE_{\alpha}(X;t)$, whereas this measure can be increasing. First, we present the following definition.
\begin{definition}
A nonnegative random variable $X$ is said to have increasing dynamic fractional generalized cumulative entropy 
(IDFCE) if $CE_{\alpha}(X;t)$ is increasing in $t$.
\end{definition}
\par
For instance, from Case (i) of Example \ref{ex:CEXt} we have that the power distribution is IDFCE. 
The following result shows that the above defined class is preserved under affine increasing transformations. 
The proof is immediate due to Proposition \ref{prop2.3}.
\begin{proposition}
Let $Y=cX+b$, where $c>0$ and $b\ge0$. If $X$ is IDFCE, then $Y$ is IDFCE.
\end{proposition}
\par
Hereafter we consider the dynamic fractional generalized cumulative entropy under the proportional reversed hazard model. 
Specifically, let $X_{\theta}$, $\theta>0$, be a random variable defined  in $(0,l)$ that satisfies the   
proportional reversed hazard model with baseline CDF $F(x)$, i.e.\ having CDF $F_{\theta}(x)=[F(x)]^{\theta}$. 
Hence, from (\ref{eq2.1*}) one has 
\begin{equation}
 CE_{\alpha}(X_\theta;t) 
 = \frac{\theta^\alpha}{\Gamma(\alpha+1)} \int_{0}^{t}    \left(\frac{F(x)}{F(t)}\right)^\theta
 \left[-\ln\left(\frac{F(x)}{F(t)}\right)\right]^{\alpha} dx, 
 \qquad  t\in (0,l).
 \label{eq:CEXtprhm}
\end{equation}
Then, it is not hard to see that in this case, for any $t\in (0,l)$ the following bounds hold:
$$
  CE_{\alpha}(X_\theta;t) 
  \left\{
 \begin{array}{ll} 
   \geq \ \theta^{\alpha} \,CE_{\alpha}(X;t) & \hbox{if } 0< \theta\leq 1 \\
   \leq \ \theta^{\alpha} \,CE_{\alpha}(X;t) & \hbox{if } \theta\geq 1. 
 \end{array} 
  \right.
$$
\par
Let us now consider two examples dealing with  the dynamic fractional generalized cumulative entropy.
\begin{example}\label{ex:exampleBD}
Let $X_\theta$ be a random variable with support $(0, \infty)$, having CDF 
$$
 F_{\theta}(x)=\left(\frac{\mu (1-e^{-(\lambda-\mu) x})}{\lambda -\mu \, e^{-(\lambda-\mu) x}}\right)^{\theta}, 
 \qquad x\in (0, \infty),
$$
for $0<\lambda<\mu$ and  $\theta>0$, and satisfying the proportional reversed hazard model. 
We remark that for $\theta\in\mathbb N$, $X_\theta$ may be viewed as the first-entrance time into the absorbing state 0 
for a linear birth-death process over $\mathbb N_0$, with birth rates $\lambda_n=\lambda\,n$ and death rates $\mu_n=\mu\,n$, 
$n\in\mathbb N_0$, having initial state $\theta\in\mathbb N$ at time 0 
(cf.\ Example 5.2 of Di Crescenzo and Ricciardi \cite{DiCrRicc2001}). 
The corresponding dynamic fractional generalized cumulative entropy, evaluated by means of (\ref{eq:CEXtprhm}), is provided 
in Figure \ref{fig:FigurePRHRM}. It is shown that it is increasing in $t$ and decreasing in $\alpha$. 
%
\begin{figure}[t]
\centering 
 \includegraphics[scale=0.4]{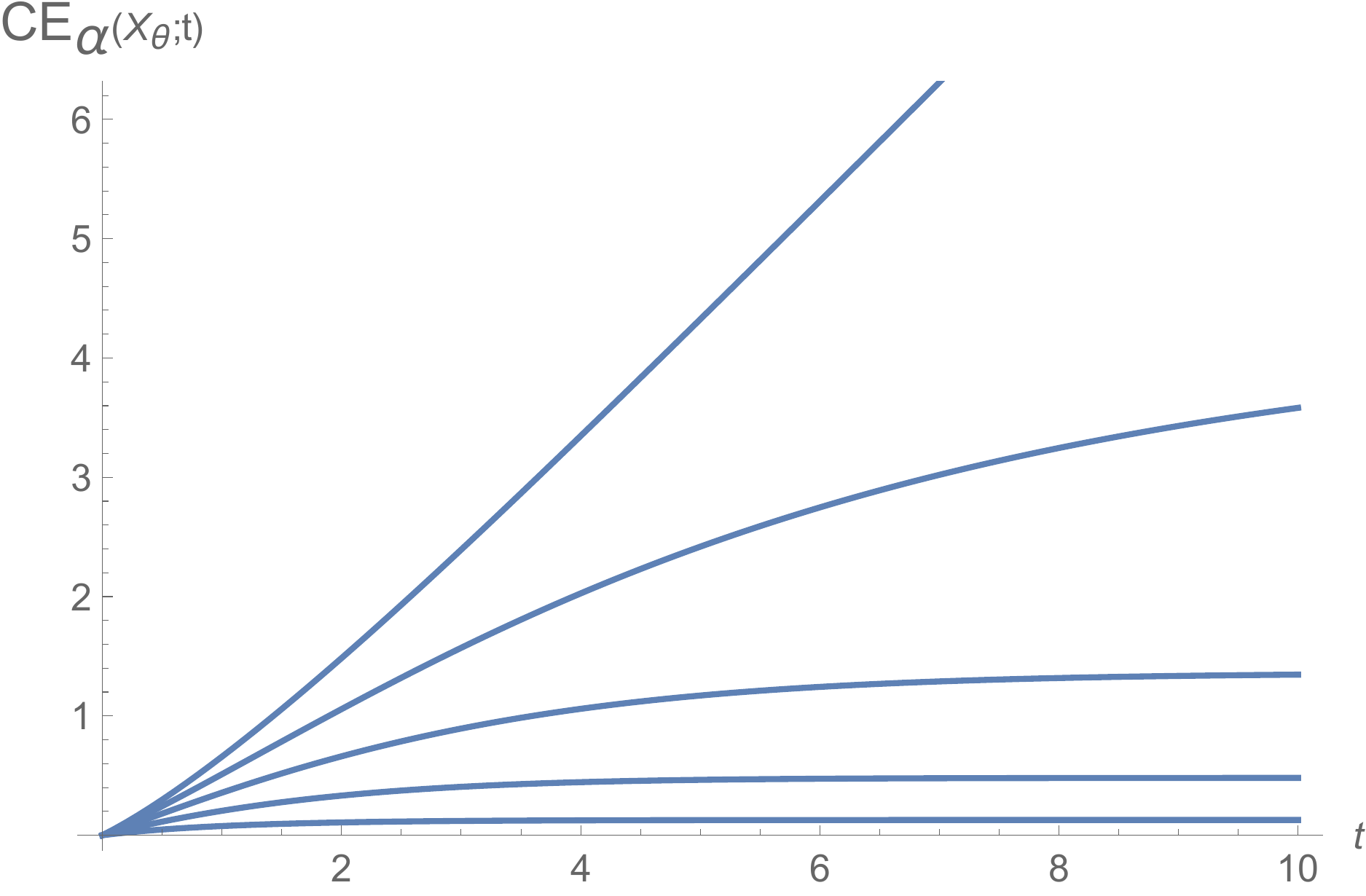}
 \;
 \includegraphics[scale=0.4]{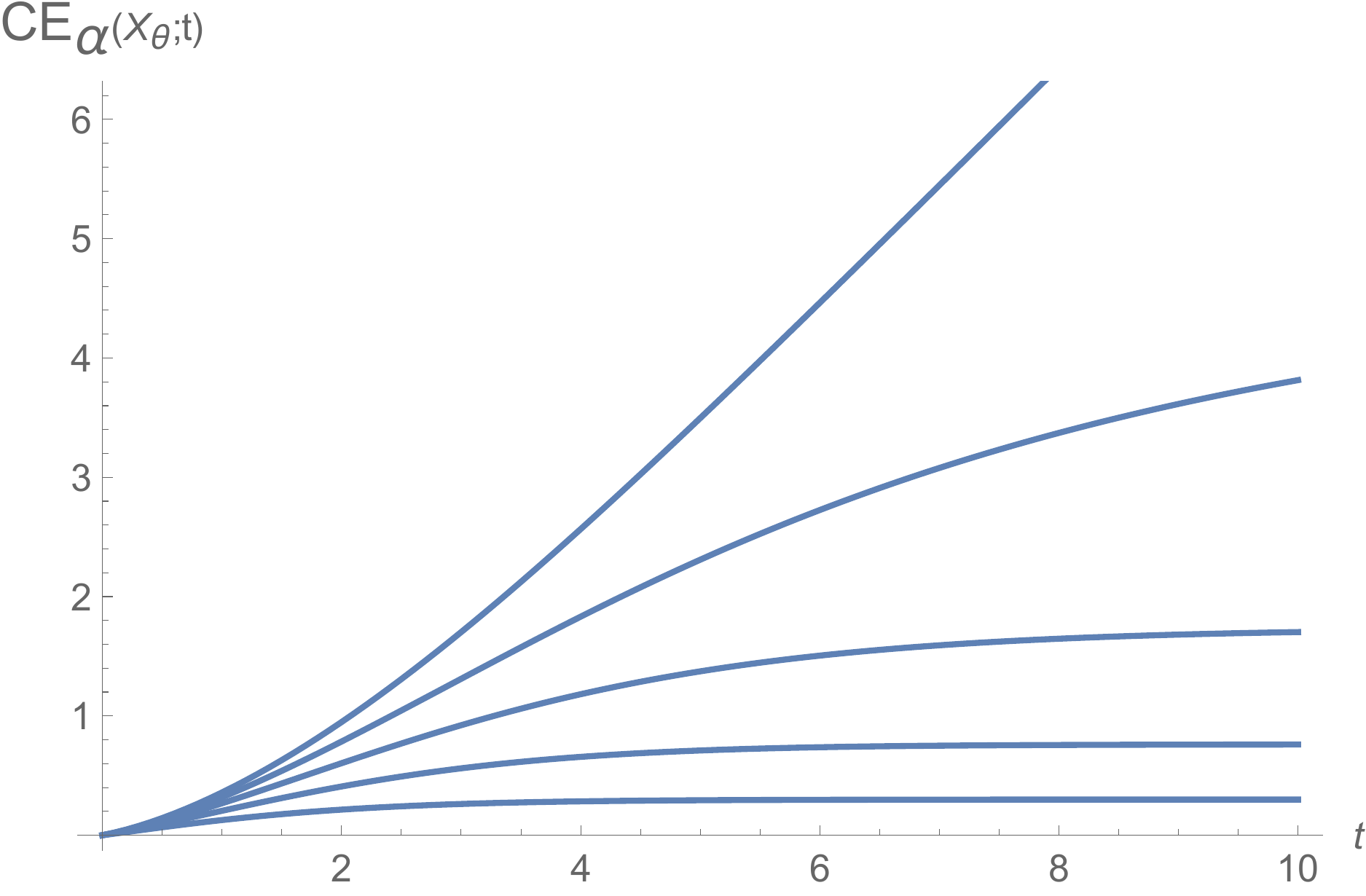}
\caption{
The function $CE_{\alpha}(X_\theta;t)$ for Example \ref{ex:exampleBD} with $\lambda=1$, $\mu=2$, and  $\theta=1$ (left) 
and $\theta=4$ (right), for $\alpha=0$, $0.2$, $0.5$, $1$, $2$ (from top to bottom). 
}
\label{fig:FigurePRHRM}
\end{figure}
\end{example}
\begin{example}\label{ex:exampleGP}
Consider the random variable $X_\theta$ having CDF
$$
 F_{\theta}(x)=\left( \frac{\lambda x}{1+\lambda x } \right)^{\theta}, 
 \qquad x\in (0, \infty),
$$
with $\lambda>0$ and  $\theta>0$. Clearly, it satisfies the proportional reversed hazard model. 
If $\theta\in\mathbb N$, then $X_\theta$ has the same distribution of the first-crossing time of the Geometric Counting Process 
with parameter $\lambda>0$  through the constant boundary $\theta$ (cf.\ Eq.\ (23) of Di Crescenzo and Pellerey \cite{DiCrPell2019}). 
Making use of  (\ref{eq:CEXtprhm}) we can evaluate its dynamic fractional generalized cumulative entropy 
(see Figure \ref{fig:FigurePRHRM2}). Also in this example, $CE_{\alpha}(X_\theta;t) $ is increasing in $t$ and decreasing in $\alpha$. 
%
\begin{figure}[t]
\centering 
 \includegraphics[scale=0.4]{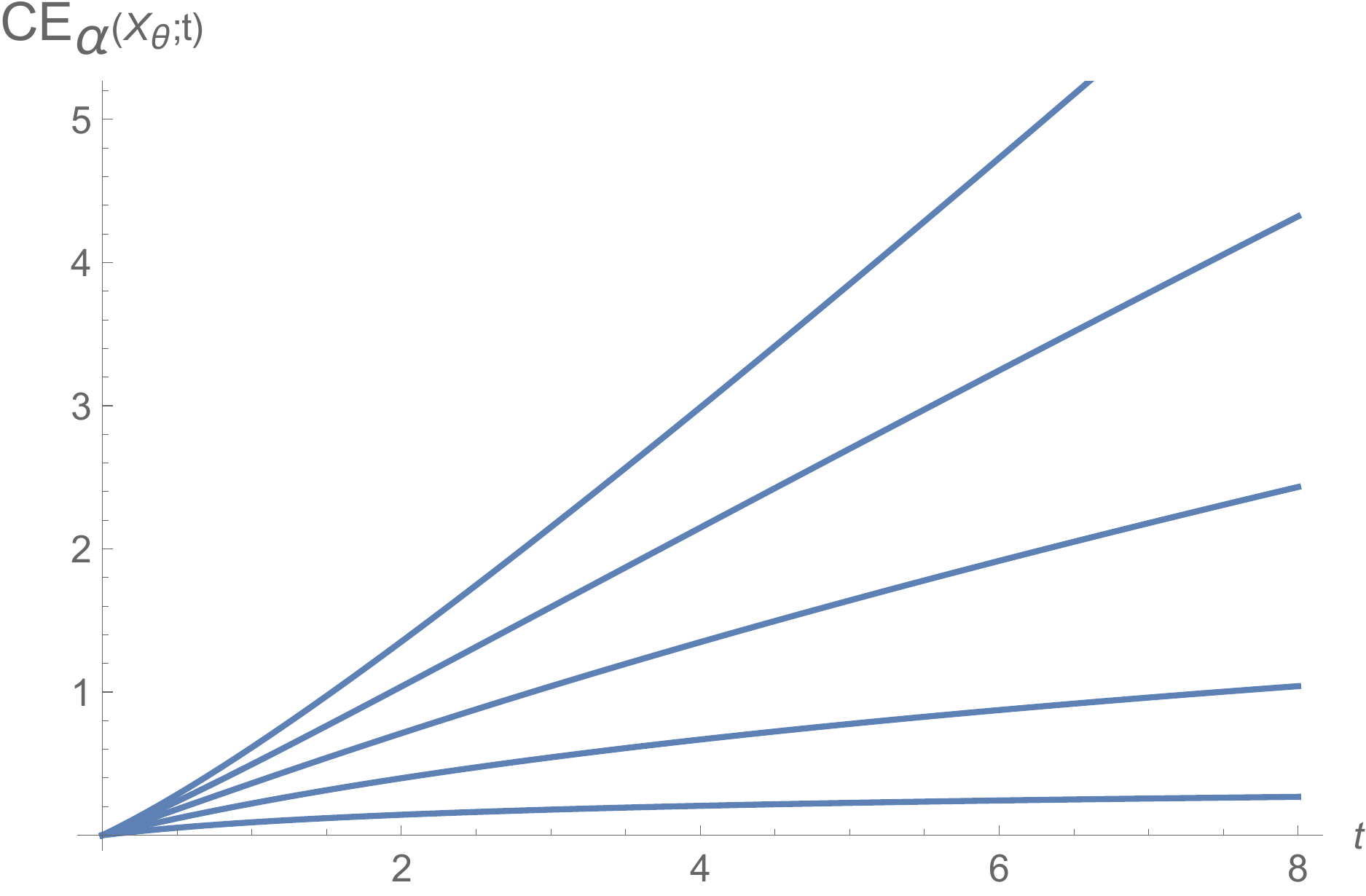}
 \;
 \includegraphics[scale=0.4]{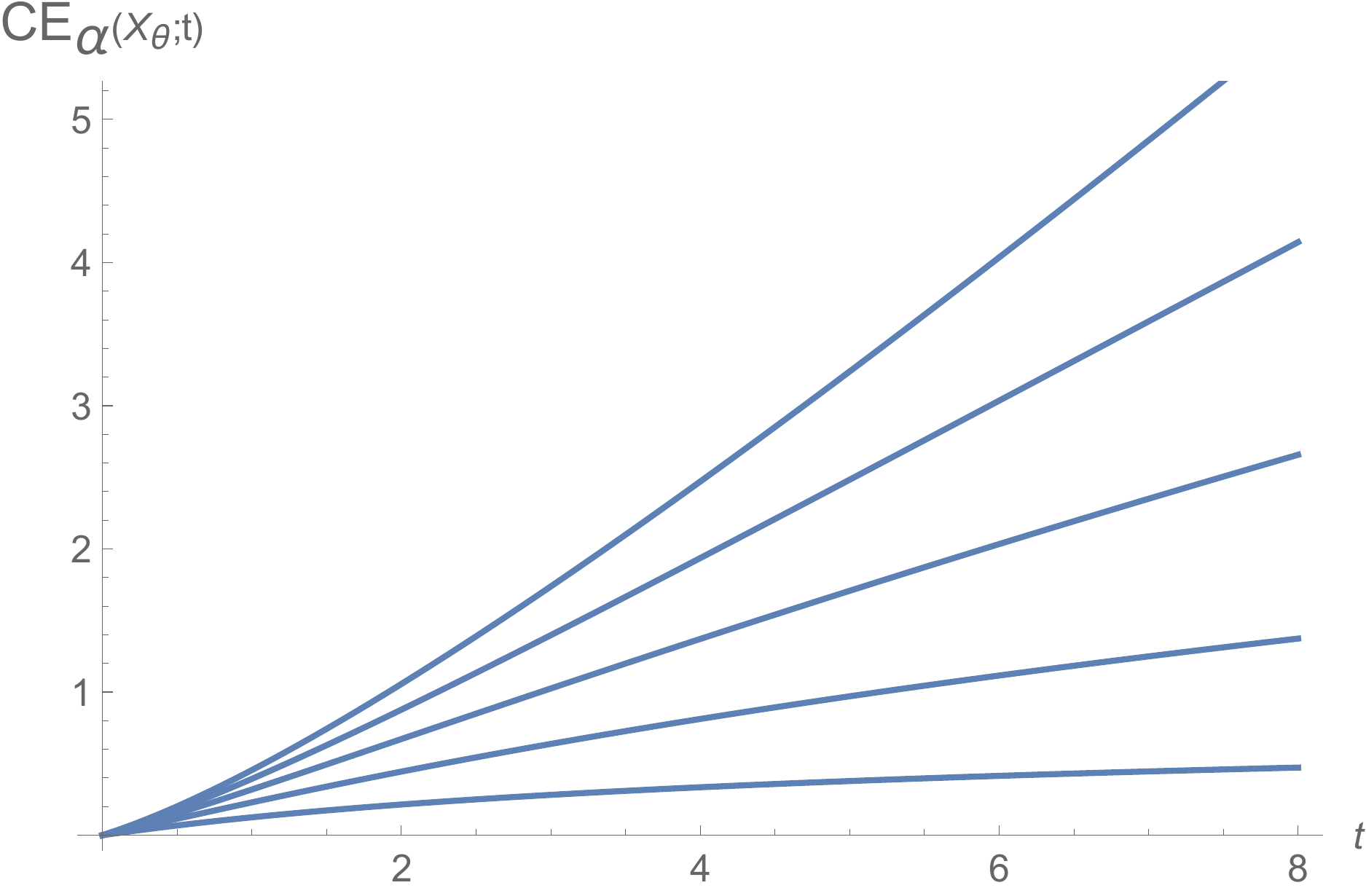}
\caption{
The function $CE_{\alpha}(X_\theta;t)$ of Example \ref{ex:exampleGP} with $\lambda=1$, and  $\theta=1$ (left) 
and $\theta=2$ (right), for $\alpha=0$, $0.2$, $0.5$, $1$, $2$ (from top to bottom). 
}
\label{fig:FigurePRHRM2}
\end{figure}
\end{example}
\par
We note that similar results can be obtained for $CRE_{\alpha}\left(X;t\right)$ under a dual model. 
Assume that $X^*_{\theta}$, $\theta>0$, is a random variable defined in $(0,l)$ which satisfies the proportional hazard model 
with baseline survival function $\bar F(x)$, i.e.\ having survival function $\bar F_{\theta}^*(x)=[\bar F(x)]^{\theta}$. 
Then, due to (\ref{eq:CREalphaXt}) the dynamic fractional generalized cumulative residual entropy of $X$ is given by 
\begin{equation}
 CRE_{\alpha}\left(X_{\theta}^*;t\right)= \frac{\theta^\alpha}{\Gamma(\alpha+1)}
 \int_{t}^{l}   \left(\frac{\bar{F}(x)}{\bar{F}(t)}\right)^{\theta} \left[- \ln\left(\frac{\bar{F}(x)}{\bar{F}(t)}\right)\right]^{\alpha} dx, 
 \qquad t\in (0,l).
 \label{eq:CREalphaXtphr}
\end{equation}
Also in this case we obtain suitable bounds for (\ref{eq:CREalphaXtphr}), i.e.
$$
  CRE_{\alpha}(X_\theta^*;t) 
  \left\{
 \begin{array}{ll} 
   \geq \ \theta^{\alpha} \,CRE_{\alpha}(X;t) & \hbox{if } 0< \theta\leq 1 \\
   \leq \ \theta^{\alpha} \,CRE_{\alpha}(X;t) & \hbox{if } \theta\geq 1. 
 \end{array} 
  \right.
$$
\par
Finally, in analogy with  (\ref{eq2.6}), we note that the normalized dynamic fractional generalized cumulative entropy can defined as 
\begin{eqnarray*}
NCE_{\alpha}(X;t)=\frac{CE_{\alpha}(X;t)}{(CE(X;t))^{\alpha}},
\qquad  t \in (0,l) \qquad (\alpha>0).  
\end{eqnarray*}
%
\section{Empirical fractional generalized cumulative entropy}
This section is devoted to the nonparametric estimate of the fractional generalized cumulative entropy.
\par
Consider a random sample $X_{1},\ldots,X_{n}$ of size $n$ from a distribution with CDF $F.$ Then, the ordered sample values denoted by $X_{1:n}\le\ldots\le X_{n:n}$ represent the order statistics of the random sample. 
As well known, the empirical distribution function based on the random sample is given by
\begin{eqnarray*}
\hat F_{n}(x) 
=\frac{1}{n} \sum_{i=1}^n I_{\{X_i\leq x\}}
= \displaystyle\left\{\begin{array}{ll} 0,
& \textrm{$x<X_{1:n},$}\\[2mm]
\displaystyle\frac{k}{n},& \textrm{$X_{k:n}\leq x<X_{k+1:n},$ \quad $(k=1,\ldots,n-1)$}\\[2mm]
1,& \textrm{$x\geq X_{n:n},$}
\end{array} \right.
\end{eqnarray*}
where $I_A$ is the indicator function of $A$, i.e.\ $I_A=1$ if $A$ is true and $I_A=0$ otherwise. 
Thus, the empirical measure of the fractional generalized cumulative entropy given by (\ref{eq2.1}) can be expressed as
\begin{eqnarray}
 CE_{\alpha}(\hat F_{n}) \!\!\!\! 
 & = &  \!\!\!\!  \frac{1}{\Gamma(\alpha+1)} 
 \int_0^l  \hat F_{n}(x) [-\ln \hat F_{n}(x)]^{\alpha} \, dx
 \nonumber \\
 & = &  \!\!\!\! \frac{1}{\Gamma(\alpha+1)}\sum_{k=1}^{n-1}V_{k+1} \left(\frac{k}{n}\right) 
 \left[-\ln\left(\frac{k}{n}\right)\right]^{\alpha},
 \label{eq4.1}
\end{eqnarray}
where 
$$
 V_{1}=X_{1:n}, \qquad V_{k+1}=X_{k+1:n}-X_{k:n},\qquad k=1,\ldots,n-1,
$$  
are the sample spacings. When $\alpha=1,$ then $CE_{\alpha}(\hat F_{n})$ reduces to the empirical cumulative entropy 
proposed in \cite{di2009cumulative} and in Di Crescenzo and Longobardi \cite{di2009IWINAC}.  
Moreover, when $\alpha$ is a positive integer then 
$CE_{\alpha}(\hat F_{n})$ identifies with the empirical generalized cumulative entropy treated in  \cite{kayal2016generalized} 
and  \cite{di2017further}. 
\par
Next, we discuss the asymptotic property of the empirical fractional generalized cumulative entropy given by (\ref{eq4.1}). 
We first shown that $CE_{\alpha}(\hat F_{n})$ converges to the actual value of the measure introduced in (\ref{eq2.1}).
\begin{proposition}\label{prop4.1}
	Let $X\in L^{p},~p>1$. Then, the empirical fractional generalized cumulative entropy converges to the fractional generalized 
	cumulative entropy almost surely. That is,  for $\alpha>0$
$$
 CE_{\alpha}(\hat F_{n})\; \; \xrightarrow[]{a. s.} \; \;  CE_{\alpha}(X), \qquad\hbox{as \ $n\rightarrow\infty.$}
$$ 
\end{proposition}
\begin{proof}
	From Glivenko-Cantelli theorem, it can be established that
$$
	\sup_{x}|\hat F_{n}(x)-F(x)|\; \; \xrightarrow{a.s.}\; \; 0,
	\qquad\hbox{as \ $n\rightarrow\infty$.}
$$  
Thus, the rest of the proof proceeds as in Theorem 9 of  \cite{rao2004cumulative}.
\end{proof}
\par
%
\begin{example}\label{ex4.2} 
For the uniformly distributed identical and independent random observations  in the interval $(0,1)$, the sample spacings 
$V_{k+1}$  follow beta distribution with parameters $1$ and $n$ with $\mathbb E(V_{k+1})=1/(n+1).$ Thus,
	\begin{eqnarray*}
	 \mathbb E\left[ CE_{\alpha}(\hat F_{n})\right]
	 =\frac{1}{\Gamma(\alpha+1)}\,\frac{1}{n+1}\sum_{k=1}^{n-1}\left(\frac{k}{n}\right)
	 \left[-\ln\left(\frac{k}{n}\right)\right]^{\alpha}
	 \equiv \mathbb E(V_{k+1}) \cdot CE_{\alpha}(U),
	\end{eqnarray*}
with $CE_{\alpha}(U)$ given in Eq.\ (\ref{eq:CEU}), 
and
\begin{eqnarray*}
	 Var\left[CE_{\alpha}(\hat F_{n})\right]
	 =\frac{1}{[\Gamma(\alpha+1)]^2}\, \frac{n}{(n+1)^{2}(n+2)}\sum_{k=1}^{n-1}\left(\frac{k}{n}\right)^{2}
	 \left[-\ln\left(\frac{k}{n}\right)\right]^{2\alpha}.
	\end{eqnarray*}
Such mean and variance are shown in Figure \ref{fig:FigureUnif} for $\alpha\in (0,3)$, with various choices of $n$. 
Similarly as Example \ref{ex4.1}, both quantities are decreasing in $\alpha$. Note that in this case we have
\begin{eqnarray*}
\lim_{n\rightarrow\infty} \mathbb E[CE_{\alpha}(\hat F_{n})]
=\frac{1}{2^{\alpha+1}} = CE_{\alpha}(X),
\qquad 
\lim_{n\rightarrow\infty}Var[CE_{\alpha}(\hat F_{n})]=0.
\end{eqnarray*}
Indeed, the considered nonparametric estimator is consistent to the fractional  generalized cumulative 
entropy when the random sample is taken from the $U(0,1)$ distribution.
%
\begin{figure}[t]
\centering 
 \includegraphics[scale=0.4]{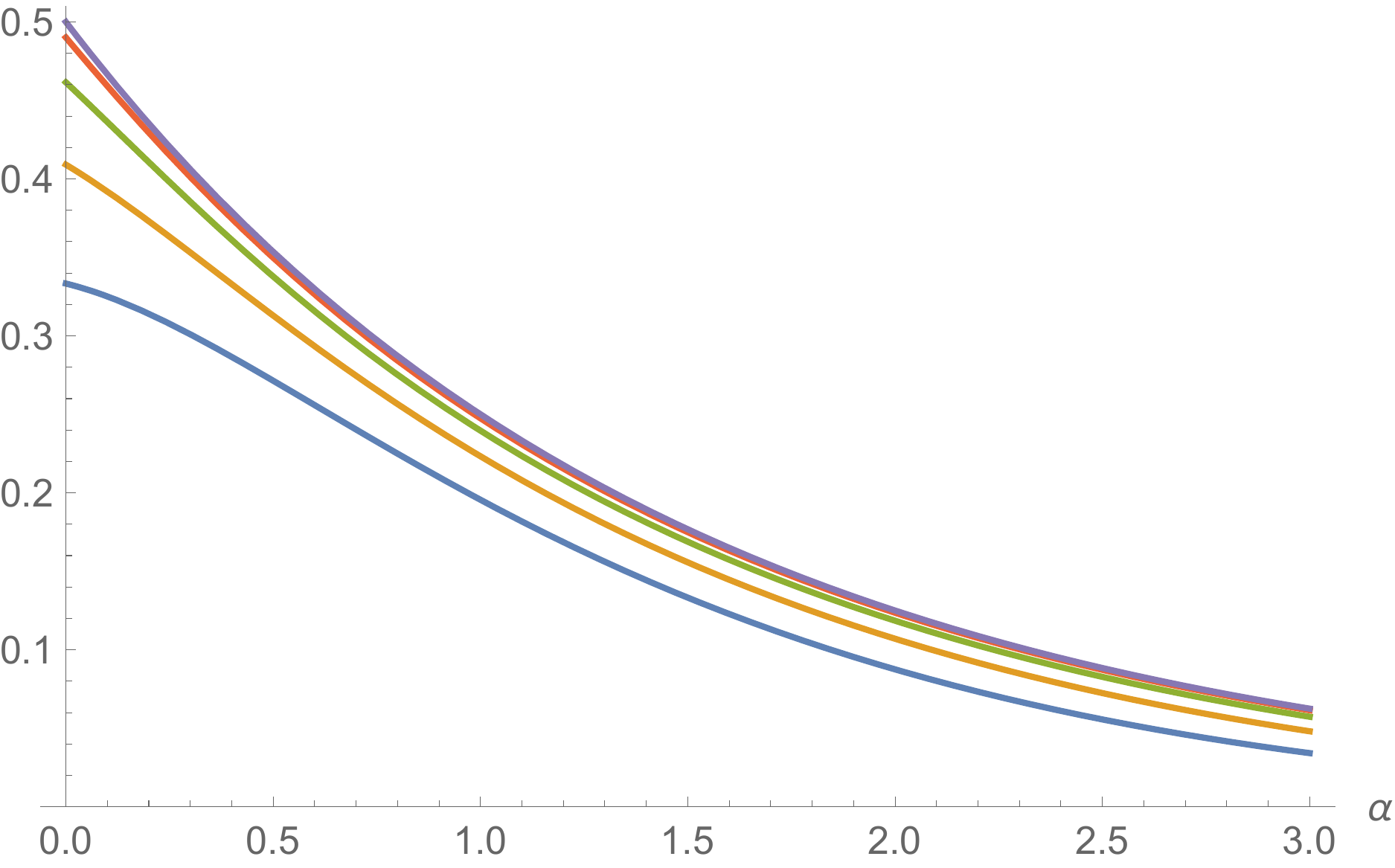}
 \;
 \includegraphics[scale=0.4]{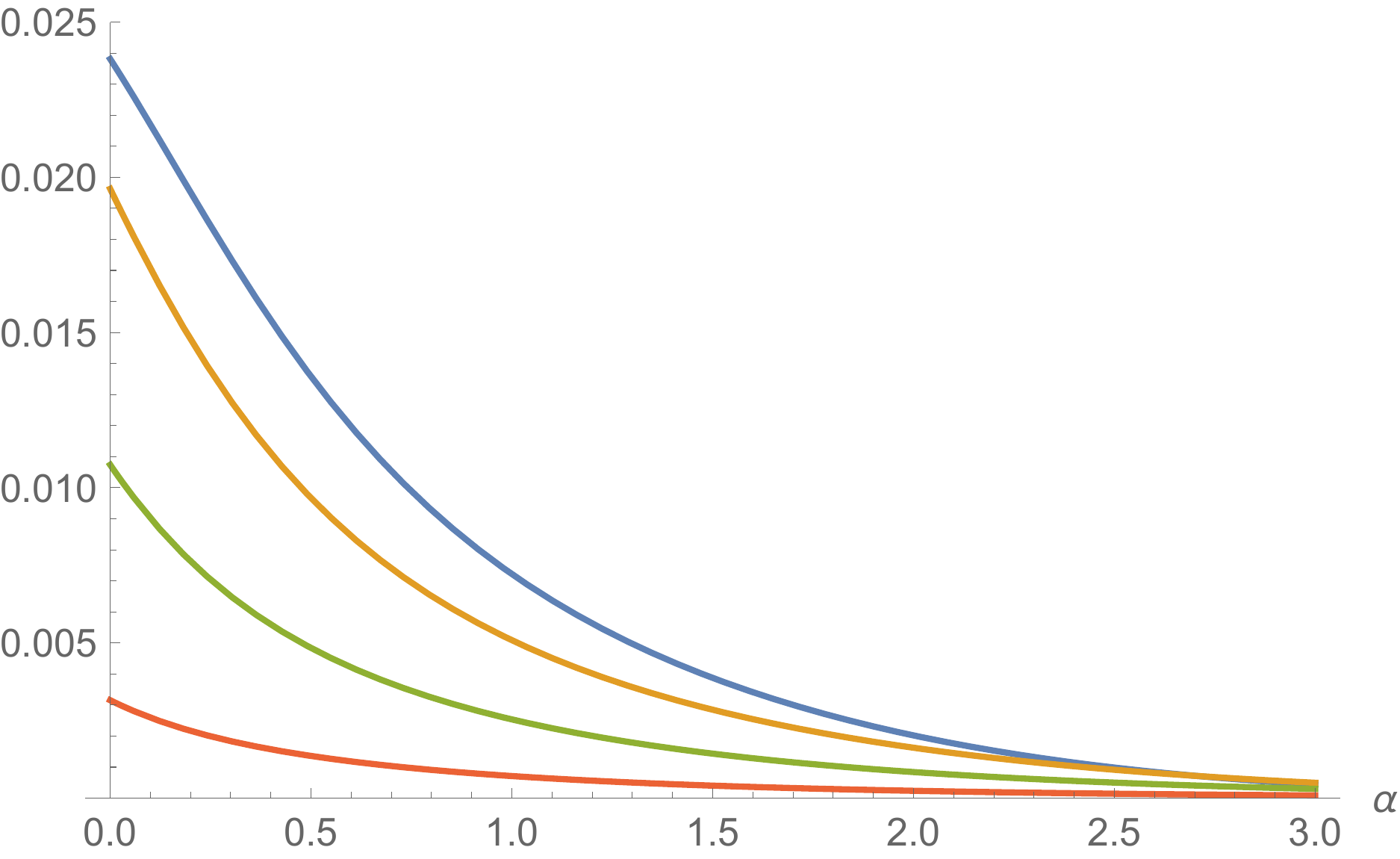}
\caption{
For Example \ref{ex4.2}, on the left:   the upper curve is the fractional generalized cumulative entropy, 
the other curves show  
$ \mathbb E\left[CE_{\alpha}(\hat F_{n})\right]$ 
for $n= 5$, $10$, $25$, $100$ (from bottom to top); on the right: $ Var\left[CE_{\alpha}(\hat F_{n})\right]$ 
for the same choices of $n$  (from top to  bottom).
}
\label{fig:FigureUnif}
\end{figure}
\end{example}
\par
Let us now discuss the stability of the empirical fractional generalized cumulative entropy, by taking as reference 
the Section 3.3 of  \cite{xiong2019fractional}.  
\begin{definition}
Let $X_{1}',\ldots,X_{n}'$ be any small deformation of the random sample $X_{1},\ldots,X_{n}$ taken from a population with CDF $F(x).$ Then, the empirical fractional generalized cumulative entropy is stable if for all $\epsilon>0$, there exists $\delta>0$ such that, 
for all $n\in \mathbb N$, 
$$
\sum_{k=1}^{n} \left|X_{k}-X_{k}'\right|<\delta
\qquad \Rightarrow \qquad 
 \left|CE_{\alpha}(\hat F_{n})-CE_{\alpha}(\hat F_{n}')\right|<\epsilon. 
$$  
\end{definition}
\par
Based on the above definition, below we present sufficient condition for the stability of $CE_{\alpha}(\hat F_{n})$.
\begin{theorem}
The empirical fractional generalized cumulative entropy of an absolutely continuous random variable $X$ is stable provided 
that $X$ has a distribution on a finite interval.
\end{theorem}
\begin{proof}
Assume that $X$ has distribution in a nonnegative finite interval.  
The empirical fractional generalized cumulative entropy is written as
\begin{eqnarray*}
CE_{\alpha}(\hat F_n) 
= \frac{1}{\Gamma(\alpha+1)} \sum_{k=1}^{n-1} (X_{k+1:n}-X_{k:n}) \hat F_n(X_{k:n})
\left[-\ln\left( \hat F_n(X_{k:n})\right)\right]^{\alpha}, 
\qquad \alpha>0.
\end{eqnarray*}
Then, the proof proceeds as for Theorem 5 of \cite{xiong2019fractional} and thus it is omitted.
\end{proof}
\par
As example, we now analyze a real data set and compute the empirical fraction generalized cumulative entropy 
for different values of $\alpha$.
\begin{example}\label{ex4.2Chowdhury}
We consider the following data set from Chowdhury {\em et al.}\ \cite{Chowdhury2017}, concerning observations 
taken from \cite{planecrashinfo} on the number of casualties in $n=44$ different plane crashes: 
$$
\begin{array}{l}
 \{3,\ 77,\ 9,\ 6,\ 14,\ 6,\ 23,\ 32,\ 18,\ 7,\ 27,\ 22,\ 10,\ 47,\ 9,\ 85,\ 7,\ 16,\ 80,\ 2,\ 8,\ 
\\
 38,\ 11,\ 12,\ 4,\ 21,\ 8,\ 44,\ 30,\ 3,\ 2,\ 19,\ 18,\ 2,\ 28,\ 8,\ 1,\ 5,\ 8,\ 1,\ 3,\ 5,\ 4,\ 3\}.
 \end{array}
$$
Based on the given dataset, we compute the values of the fractional generalized cumulative entropy (\ref{eq4.1}), 
shown in Figure \ref{fig:CEmp}. Hence, we deal with a linear combination of terms of the type  
$$
 \varphi(\alpha; x):=\frac{x^{\alpha}}{\Gamma(\alpha+1)}, \qquad \alpha>0, \;\; x=-\ln \left(\frac{k}{n}\right)>0.
$$
Note that if $0< x \leq e^{-\gamma}=0.5615...$ (where $\gamma$ is the Euler-Mascheroni constant) then 
the function $\varphi(\alpha; x)$ is decreasing in $\alpha$; moreover if 
$0< x \leq \exp\big\{-\frac{1}{2} \big(2 \gamma + \sqrt{(2/3)}\,\pi\big)\big\}=0.1557...$ then $\varphi(\alpha; x)$ 
is convex in $\alpha$. If the sample spacings are slowly varying, since the larger coefficients in the sum on the 
right-hand side of (\ref{eq4.1}) are given by large $k$, and thus for $x$ close to 0, then in the linear combination 
for the fractional generalized cumulative entropy the prevailing terms are decreasing convex in $\alpha$. 
This remark justifies its form in the left plot of Figure \ref{fig:CEmp}, where it is shown as a function of $\alpha$. 
Furthermore, when it is treated as a function of $n$, i.e.\ referring to the first $n$ data of the sample, 
the right plot of Figure \ref{fig:CEmp} shows a jagged trend, that is smoother for larger values of $\alpha.$
\end{example}
\begin{figure}[t]
\center
 \includegraphics[scale=0.4]{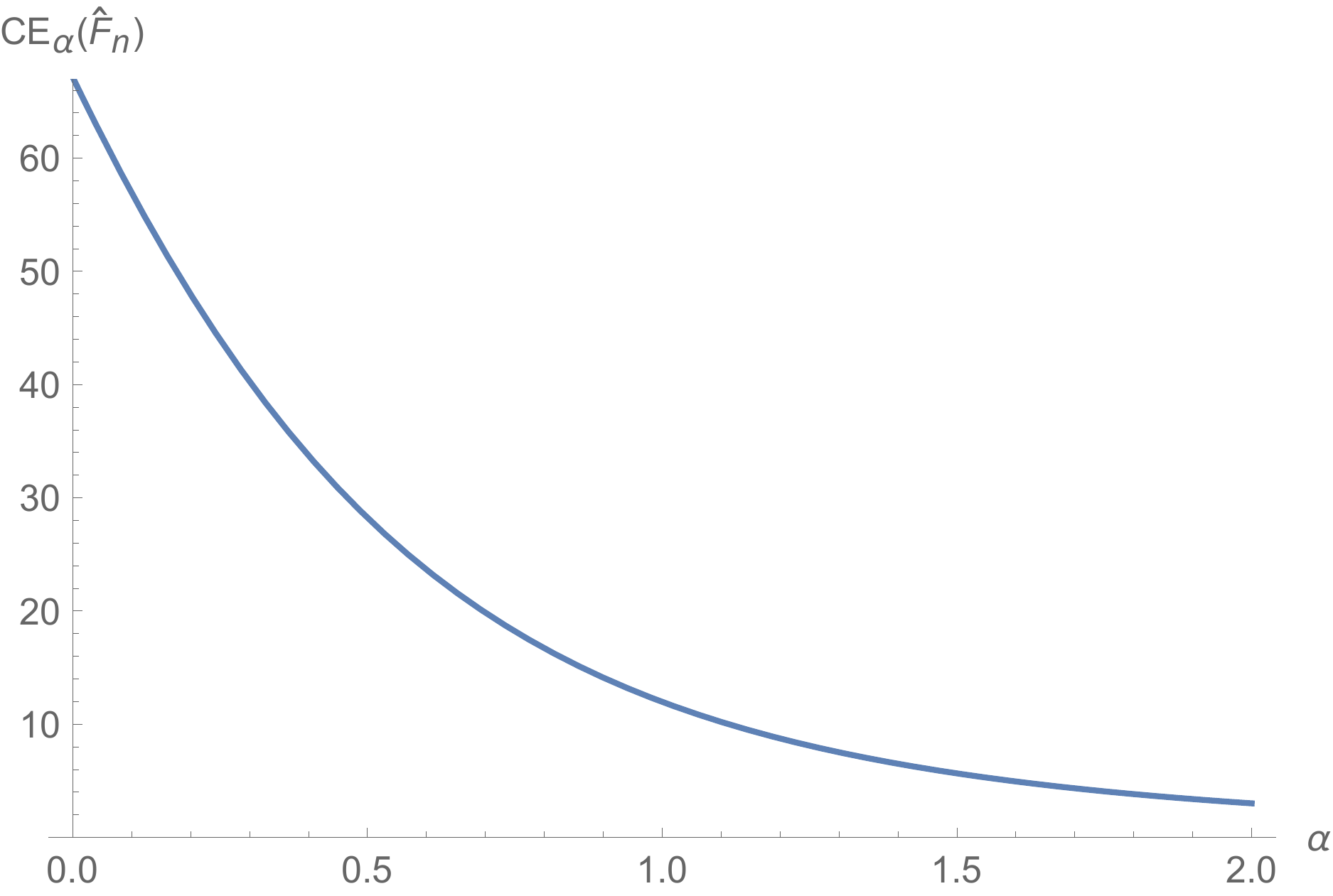}
 \;
 \includegraphics[scale=0.4]{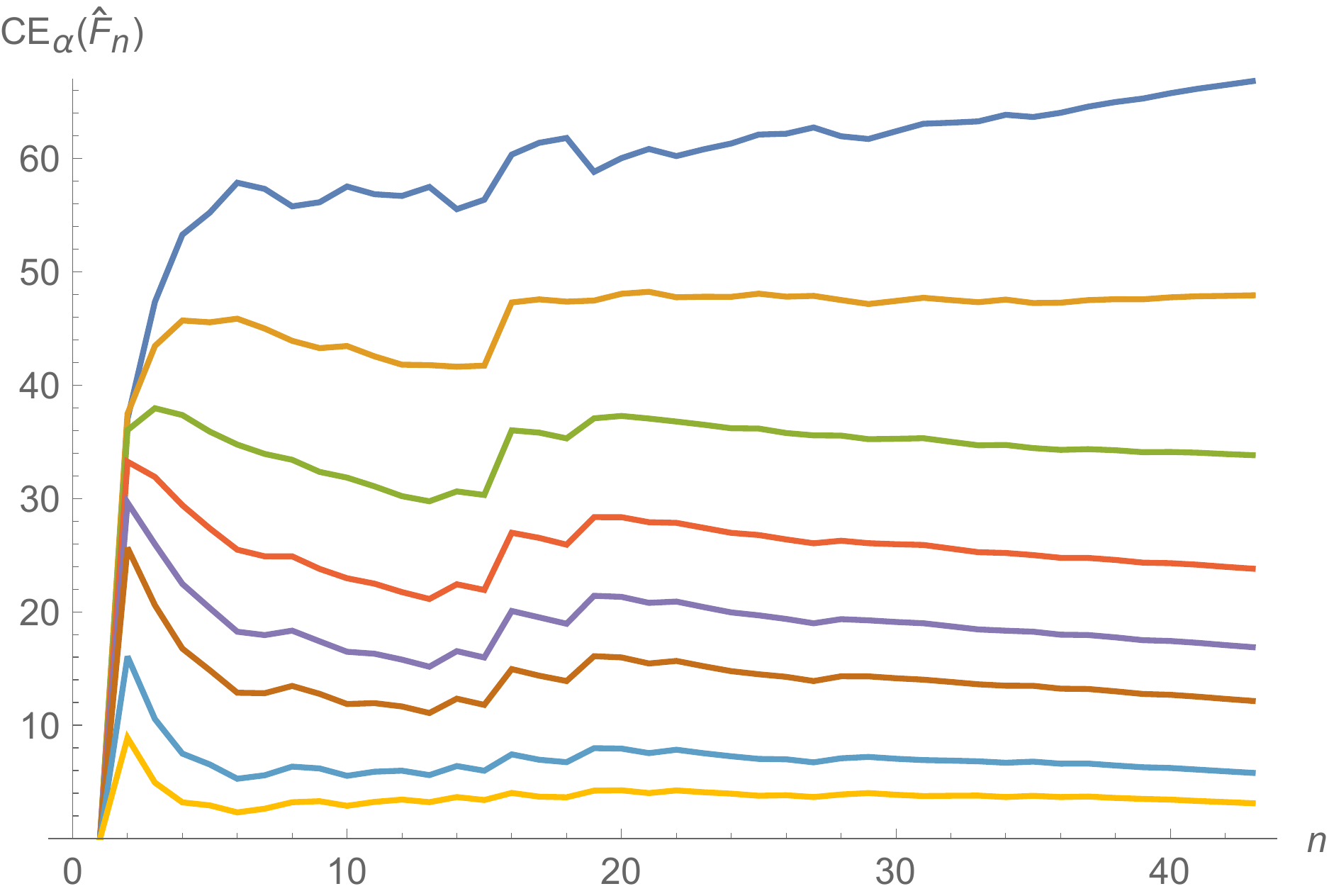}

\caption{ 
The values of the fractional generalized cumulative entropy for the data set of Example \ref{ex4.2Chowdhury}. On the left: for $0\leq \alpha\leq 2$ 
and $n=44$. On the right: for $n$ varying, with reference to the first $n$ data of the sample, with 
$\alpha=0,~0.2,~0.4,~0.6,~0.8,~1,~1.5$ and $2$ (from top to bottom). 
}
\label{fig:CEmp}
\end{figure}
%
\subsection{Exponential distribution}
In this section, we consider the special case in which the i.i.d.\ random observations 
are available from the exponential distribution.
We first present a central limit theorem  for the empirical fractional generalized 
cumulative entropy. 
\begin{proposition}
	Let $X_{1},\ldots,X_{n}$ be a random sample from the exponential distribution with parameter $\lambda$. Then, for any $\alpha >0$
$$
 Z_{n}:=\frac{CE_{\alpha}(\hat F_{n})-\mathbb E[CE_{\alpha}(\hat F_{n})]}
   { \left( Var[CE_{\alpha}(\hat F_{n})] \right)^{1/2}} 
	\quad \rightarrow \quad {\cal N}(0,1)
$$ 
in distribution as $n\rightarrow\infty.$
\end{proposition}
\begin{proof}
	From \cite{di2009cumulative}, we note that the empirical fractional generalized cumulative entropy 
	can be expressed as the sum of independent exponential random variables $U_{k}$ with mean
	\begin{eqnarray*}
	\mathbb E[U_{k}]=\frac{1}{\Gamma(\alpha+1)}\,
	\frac{k}{\lambda n(n-k)}\left[-\ln\left(\frac{k}{n}\right)\right]^{\alpha}.
	\end{eqnarray*}
The rest of the proof follows using similar arguments in \cite{di2017further}. Thus, it is omitted.
\end{proof}
\begin{example}\label{ex4.1} 
Consider a random sample $X_{1},\ldots,X_{n}$ from the exponential distribution with parameter $\lambda.$ 
Since the sample spacings are independent, thus, $V_{k+1}$ follows the exponential distribution with 
parameter $\lambda(n-k).$ So, from (\ref{eq4.1}) the expectation and variance of the empirical fractional 
generalized cumulative entropy are respectively obtained as
		\begin{eqnarray*}
		 \mathbb E\left[CE_{\alpha}(\hat F_{n})\right]
		 =\frac{1}{\Gamma(\alpha+1)}\,\frac{1}{\lambda}\sum_{k=1}^{n-1}\frac{1}{n-k}
		 \left(\frac{k}{n}\right)\left[-\ln\left(\frac{k}{n}\right)\right]^{\alpha},
		\end{eqnarray*}
		and
\begin{eqnarray*}
		 Var\left[CE_{\alpha}(\hat F_{n})\right]
		 =\frac{1}{[\Gamma(\alpha+1)]^2}\,
		 \frac{1}{\lambda^2}\sum_{k=1}^{n-1}\frac{1}{(n-k)^2}\left(\frac{k}{n}\right)^2\left[-\ln\left(\frac{k}{n}\right)\right]^{2\alpha}.
\end{eqnarray*}
Figure \ref{fig:FigureEsp} shows the above quantities as a function of $\alpha$, for some choices of $n$. 
In particular, both mean and variance are decreasing in $\alpha$. 
Moreover, it is shown that the mean $\mathbb E\left[CE_{\alpha}(\hat F_{n})\right]$ 
approaches  the fractional generalized cumulative entropy as $n$ grows, more rapidly for larger $\alpha$. 
%
%
\begin{figure}[t]
\centering 
 \includegraphics[scale=0.4]{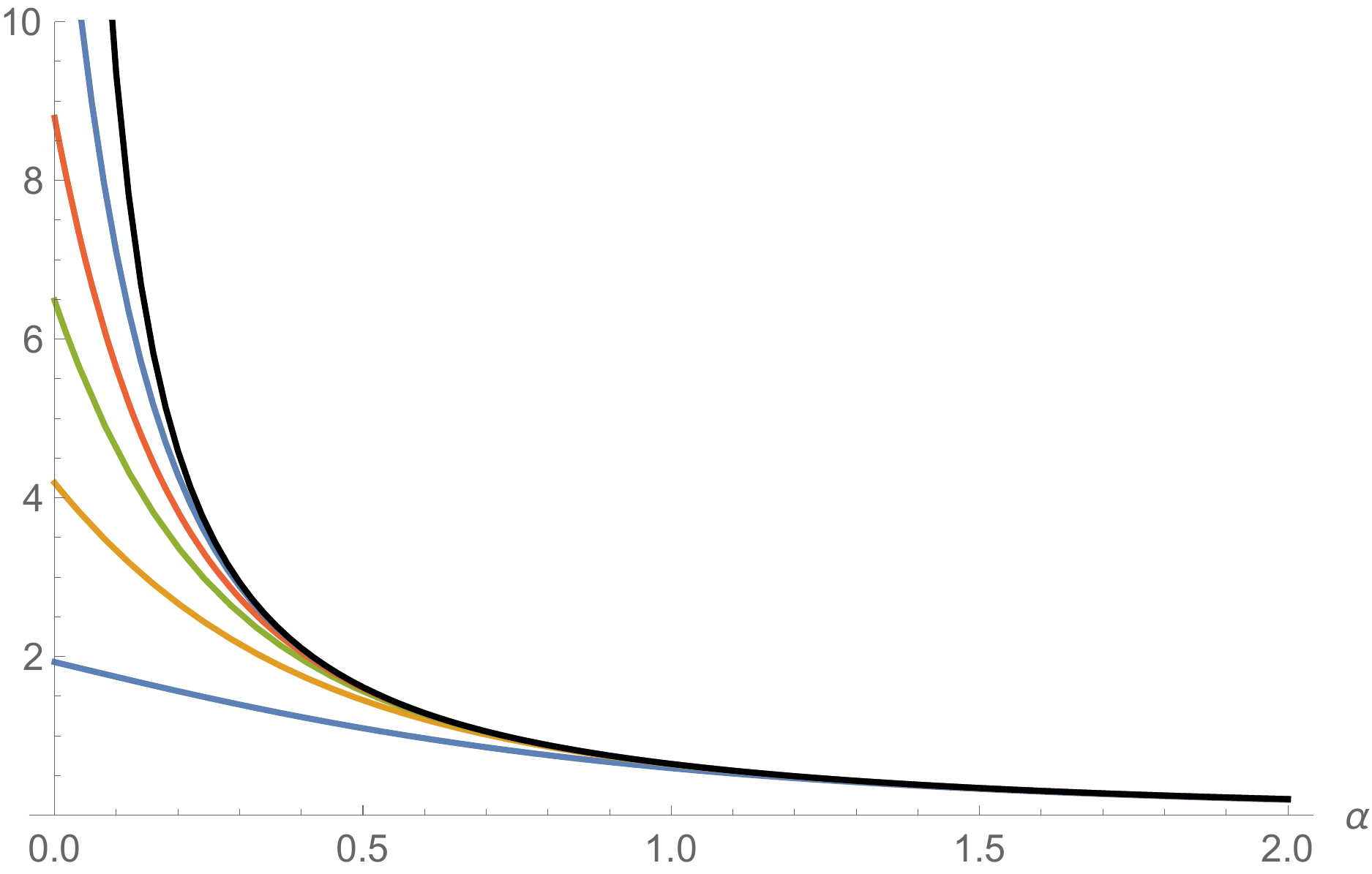}
 \;
 \includegraphics[scale=0.4]{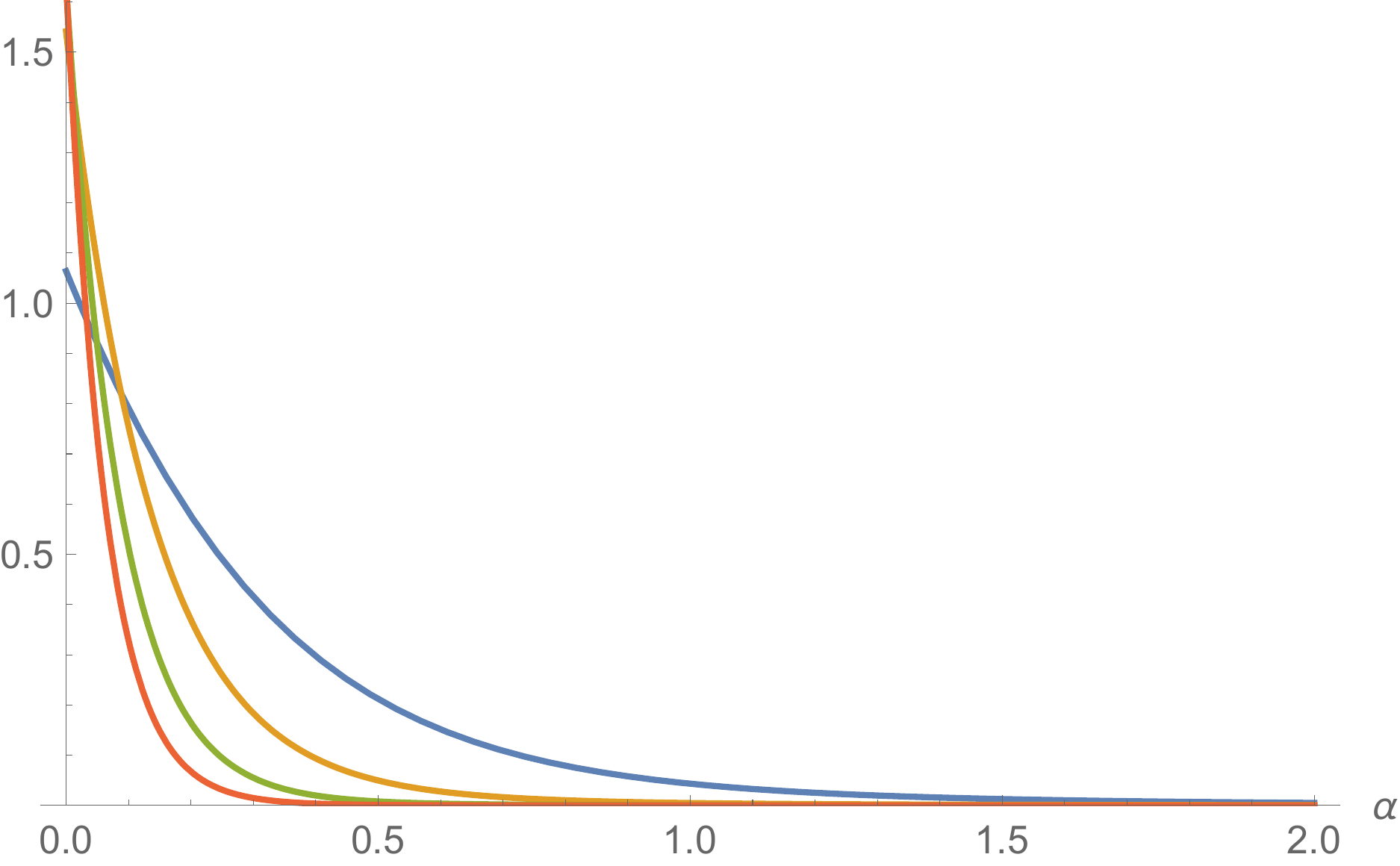}
\caption{
With reference to Example \ref{ex4.1}, for $\lambda=1$, on the left: the upper curve is the fractional generalized 
cumulative entropy, the other curves give $ \mathbb E\left[CE_{\alpha}(\hat F_{n})\right]$ 
for $n=10^h$, $h=1$, $2$, $3$, $4$, $6$ (from bottom to top); on the right: $ Var\left[CE_{\alpha}(\hat F_{n})\right]$ 
for $n=10^h$, $h=1$, $2$, $3$, $4$  (from top to  bottom in proximity of $\alpha=0.5$).
}
\label{fig:FigureEsp}
\end{figure}
\end{example}

\section{Concluding remarks}
In this paper, we  defined the fractional generalized cumulative entropy and its dynamic version. 
Moreover, we provided an interesting link with fractional integrals of generic order 
$\alpha>0$, which could create new research ideas in the theory of fractional calculus. 
Various properties including bounds and ordering results have been studied. It is shown that the usual stochastic 
ordering does not imply the ordering between the considered entropies. 
However, we have shown that the dispersive order implies the ordering of the  considered measure. Thus, the fractional 
generalized cumulative entropy  actually constitutes a variability measure. This fact discloses the possibility of applications in risk theory 
involving the proportional hazards model, for instance along the line addressed by Psarrakos and Sordo \cite{PsarrakosSordo}. 
\par
A nonparametric estimator of the fractional measure has been  proposed based on the empirical distribution function. 
Various statistical properties of the empirical fractional generalized cumulative entropy have been studied, 
including asymptotic results for large samples. A stability criteria of the proposed measure has been studied, too. 
Finally, we focus on the convergence of the estimator 
and on the central limit theorem when a random sample is taken from the 
exponential distribution. 
%
\subsection*{Acknowledgements}
%
Antonio Di Crescenzo and Alessandra Meoli are members of the research group GNCS of INdAM (Istituto Nazionale di Alta Matematica). 
This research is partially supported by MIUR - PRIN 2017, project `Stochastic Models for Complex Systems', 
No.\ 2017JFFHSH. 
Suchandan Kayal gratefully acknowledges the partial financial support for this research work under a grant MTR/2018/000350, SERB, India.
%

%
\end{document}